\documentclass[reqno,final]{siamltex}

\usepackage{latexsym}
\usepackage{amsmath,amssymb, amsxtra}
\usepackage[english]{babel}
\usepackage{enumerate}
\usepackage{graphicx}
\usepackage{color}
\newcommand{\vel}{\mathcal{V}}

\newcommand{\real}{\mathbb{R}}
\newcommand{\prob}{{\cal  P}}
\newcommand{\mes}{{\cal  M}}
\newcommand{\p}[1]{\left(#1\right)}

\newcommand{\abs}[1]{\left|#1\right|}
\newcommand{\norm}[1]{\left\|#1\right\|}
\newcommand{\der}[2]{\frac{\partial #1}{\partial #2} }
\newcommand{\commentout}[1]{}

\newcommand{\bs}{\backslash} 
\newcommand{\qed}{ \ensuremath{\Box}}

\newtheorem{remark}{Remark}

\newcommand{\be}{\begin{equation}}
\newcommand{\ee}{\end{equation}}
\newcommand{\ba}{\begin{eqnarray}}
\newcommand{\ea}{\end{eqnarray}}

\def\1{\mathbb{I}}

\def\R{\mathbb{R}}
\def\RR{\mathbb{R}}

\def\a{\alpha}

\def\calW{\mathcal{W}}
\def\calB{\mathcal{B}}

\def\rmL{\mathrm{L}}

\title{Characterization of radially symmetric finite time blowup in multidimensional aggregation equations\thanks{This work was supported by NSF grants DMS-0714945, DMS-0907931 and DMS-1109805.   The date of this manuscript is Jan.1, 2011.}}
\author{Andrea L. Bertozzi\thanks{Department of Mathematics, University of California Los Angeles, Los Angeles, CA 90095 ({\tt bertozzi@math.ucla.edu}).}
        \and John B. Garnett \thanks{Department of Mathematics, University of California Los Angeles, Los Angeles, CA 90095 {\tt (jbg@math.ucla.edu)}} \and Thomas Laurent \thanks{Department of Mathematics, University of California Riverside, Riverside CA 92521}  ({\tt laurent@math.ucr.edu}) }

\begin{document} 

 \maketitle

\begin{abstract}
This paper studies the transport of a mass $\mu$ in $\real^d, d \geq 2,$  by a flow field $v= -\nabla K*\mu$.  
We focus on kernels $K=|x|^\alpha/ \alpha$ for $2-d\leq  \alpha<2$ for which the smooth densities are known to develop singularities in finite time.  
For this range
 we prove the existence for all time of radially symmetric measure solutions
 that are monotone decreasing as a function of the radius, thus allowing for continuation of the solution past the blowup time.
The monotone constraint on the data is consistent with the typical blowup profiles observed in recent numerical studies of these singularities.  
We prove monotonicity is preserved for all time, even after blowup, in contrast to the case  $\alpha >2$
 where radially symmetric solutions are  known to  lose monotonicity.
In the case of the Newtonian potential ($\alpha=2-d$), under the assumption
 of radial symmetry the equation can be transformed
 into the inviscid Burgers equation on a half line.  This enables us to prove preservation of monotonicity using the classical theory of conservation laws.
 In the case $2 -d < \alpha < 2$ and at the critical exponent $p$
we exhibit initial data in $L^p$ for which the solution immediately develops a Dirac mass singularity.  This extends recent work on the local
ill-posedness of solutions at the critical exponent.
\end{abstract}
\date\today


\section{Introduction}

\

This manuscript considers the problem of dynamic nonlocal aggregation equations of the form \begin{equation}
\frac{\partial \rho}{\partial t}  - \text{div}(\rho \nabla K*\rho) = 0
\label{aggeq}
\end{equation}
in $\real^d$ for $d \ge 2$.
This problem has been a very active area of
research in the literature \cite{BenedettoCagliotiPulvirenti97,
BB,
BCL,
BL,
BLR,
BW,
BCM,
BDP,
BV,
BCMo,
BCKSV,
BuCM,
BDi,
CDFLS,
cr09,
DP,
GP,
HJD,
HP,
HPPRL,HB,
Keller-Segel-70,
L,
LR1,
LR2,
LT,
LZ,
ME,
MCO,
TB,
TBL,
tosc:gran:00}. 
These models arise in a number of applications including
 aggregation in materials science \cite{HP,
HPPRL,Poupaud1, Poupaud2}, cooperative control \cite{GP}, granular flow \cite{Carrillo-McCann-Villani03,Carrillo-McCann-Villani06,tosc:gran:00}, biological swarming models \cite{ME,TB,TBL}, evolution of vortex densities in superconductors
\cite{Weinan:1994,
AS,AMS,DZ, Masmoudi} and bacterial chemotaxis \cite{BCKSV,Keller-Segel-70,BCM,BDP}.
A body of recent work has focused on the problem of finite time singularities and
local vs global well-posedness in multiple space dimension for both
the inviscid case (\ref{aggeq}) \cite{BB,BCL,BL,BLR,BV,BS,CDFLS,HJD,HB,L}
and the cases with various kinds of diffusion \cite{BRB,BCM,BS,LR1,LR2}.  
The highly studied Keller-Segel problem  typically has a Newtonian potential and linear diffusion.
For the pure transport problem  (\ref{aggeq}),
of particular interest is the transition from smooth solutions to weak and measure solutions
with mass concentration.  This paper presents a general framework for
radially symmetric solutions that blowup in finite time in which
the initial data decreases monotonically from the origin.
This paper differentiates itself from the previous work in that it considers continuation of the solution as a measure past the initial singularity for the range of $1>\alpha>2-d$.
Prior work  on measure solutions considers the case $\alpha\geq 1$ in general dimension \cite{CDFLS} and the Newtonian case with a defect measure in two dimensions \cite{Poupaud2}.
The monotone constraint is typical of the local structure of the solution near the blowup time as shown in numerical simulations.
This paper presents rigorous theory for solutions with such structure, showing that the nonlocal evolution preserves this global structure.
Much work has been done on related problems in fluid dynamics to understand measure solutions of active scalar equations \cite{Delort,DM,LNS07,Poupaud2,Zheng}.  
The results of this paper may provide further insight for those problems.

In the case of power-law kernels e.g. $K(x) = |x|^\alpha/ \alpha$, it is already known that the critical power is $\alpha=2$. For $\alpha<2$ finite time singularities always arise and for $\alpha\geq 2 $
solutions stay smooth for all time if the initial data is smooth \cite{BCL,BL,CDFLS}. In the case of finite time blowup, i.e. $\alpha<2$, it has been observed numerically \cite{HB,HB1,Hthesis} that, starting with a smooth radially symmetric  initial data which is monotone decreasing as a function of the radius, the solution evolves so that the monotonicity is preserved and at some finite time develops a power-law singularity at the origin. 
The power is sufficiently singular that, based on the results in this paper, the solution produces an instantaneous mass
concentration after the initial singularity. To make this rigorous we must develop a general theory for such singular solutions.
This paper addresses this particular class of measure solutions, namely those with a radially symmetric decreasing profile and possibly  a Dirac mass at the origin. We prove that this structure is  maintained for all time when $\alpha<2$.
We conjecture that this class of radially symmetric decreasing solutions including a Dirac mass at the origin describes well the local behavior at the blow up time of general non radially symmetric solutions.   We also note that  for $\alpha>2$ it has been observed in numerical simulations \cite{HB1,Hthesis}  that   monotone decreasing structures are not preserved: indeed there is an attracting solution
of the form of a collapsing delta-ring which  causes mass to
collect on the ring  during the collapse thereby destroying any initial monotone property of the solution \cite{HB1,Hthesis}. 

We note that our work fits nicely between the general measure
theory in \cite{CDFLS} (for $\alpha\geq 1$) and previous works which consider more regular classes of weak solutions including $L^\infty$ \cite{BB,BCL} and $L^p$ \cite{BLR}.  For $L^\infty$ and $L^p$ solutions
we typically have only local well-posedness whereas the measure solutions have global well-posedness. 
Our work extends the global existence results in \cite{CDFLS} to the case of more singular kernels with power $\a\geq 2-d$, for the special case of monotone decreasing radially symmetric  measure solutions.  This includes that of the Newtonian potential, which is discussed separately in the next paragraph.  Uniqueness of solutions for $2-d<\a<1$ is still an open problem.

 In two space dimensions, when $K(x)= \log(|x|)$ (i.e. $K$ is the  Newtonian
potential),  the aggregation equation  arises as a model for the
 evolution of vortex densities in superconductors
\cite{Weinan:1994,SandierSerfaty, SandierSerfatybook,LinZhang,
AS,AMS,Mainini,DZ, Masmoudi}, and also in models for adhesion dynamics
\cite{Poupaud1, Poupaud2}.  In these models singularities are known to
appear in finite time, and the question of interest is how to continue the
solution after the initial  formation of singularities.  Since these
singularities are expected to be Dirac masses one has to consider measure
solutions. Unfortunately, due to the very singular behavior of the Newtonian
potential at the origin, most of the results to date concern  the
existence of measure solutions 
 which contain an error term (a defect measure) compared to the original
equation \cite{DZ,Poupaud1,AMS}. Also uniqueness is lacking in these works. 

In this paper we consider the Newtonian case in all dimensions and show
that for general radially symmetric data there is no need to consider a defect measure
because the symmetry allows the problem to be reduced to a form of the inviscid Burgers equation on the half line, for which many things are known.  In particular, the case of radially symmetric
monotone decreasing densities maps to classical Lipschitz solutions of the inviscid
Burgers equation, without shocks, allowing us to prove such solutions exist and are unique for all time.
For the non-monotone case, shocks can form, corresponding to mass concentrations along
spherical shells, and their evolution is not immediately well-defined, due to a jump in the velocity field at the shell.  However, one can
use the classical weak solution theory for Burgers equation to define a jump condition through a weak form
of the evolution equation, or through some other convention.  If we use the classical Burgers 
shock solution then the solution is unique since it automatically satisfies the Lax entropy condition.
The more singular case of signed measures can also be studied in this framework, in which
case one must consider rarefaction solutions as well as shocks.

Going back to the case $K(x)=\abs{x}^\alpha/\alpha$, $2-d<\alpha<2$, recent computational results \cite{HB} show that the initial finite
time blowup from radially symmetric data has a simple self-similar form
in which the power-laws of the similarity solution have anomalous scaling
but the shape of the similarity solution has a simple monotonically decreasing
structure with powerlaw tail.  The power in the tail determines the
degree of singularity of the solution at the initial blowup time -
we observe that at the initial blowup time the solution leaves $L^\infty$ but remains in some $L^p$ spaces and does not
concentrate mass. 
This result prompted a careful study
of the well-posedness of the equation in $L^p$ spaces \cite{BLR}.
In that paper it was proved that for a given interaction kernel, there
exists a critical $L^p$ space such that the problem is locally well-posed
for $p>p_c$.  Moreover it was proved in \cite{BLR}
that the power $p_c$ is sharp for $K=|x|$, i.e.
the problem is locally ill-posed for $p<p_c$. 
Some of these results, in particular the critical $L^{p}$ space,
have been extended to general power-law kernels in \cite{HJD}. In the present work we examine the mechanism by which initial data in
the critical space $L^{p_c}$  leave instantaneously this space.
Taking advantage of our existence theory for radially symmetric decreasing measure solutions when $2-d<\alpha<2$, we exhibit a large class of radially symmetric decreasing initial data in $L^{p_c}$ for which a Dirac mass forms  instantaneously in the solution. This is a natural extension of the results in \cite{BLR} and \cite{HJD}.

This paper is organized as follows: below we review the mathematical notation and basic functional analysis used in this paper.  Section~\ref{genrad} develops a general existence theory for radially symmetric solutions with the monotonicity constraint.    
Section~\ref{instant} proves instantaneous mass concentration for the critical $L^p$ spaces.
Section~\ref{Newtonian} considers the case of the Newtonian potential, for which 
we can show that radial symmetry results in a transformation of the nonlocal problem to the inviscid Burgers equation on a half line. Section~\ref{conclusions} summarizes the results and discusses some open problems.
In the appendix we derive some background theory of ordinary differential equations needed for the proofs in this paper and not derived in standard references (although the arguments
are similar to standard methods).

\subsection{Mathematical formulations and notation}

\

The aggregation equation, for smooth solutions, in Eulerian coordinates, is
\begin{align}
&\frac{\partial \rho}{\partial t} + \text{div} (\rho v)=0  \label{agg1}\\
&v(x,t)=-\nabla K*\rho.
\label{agg2}
\end{align}
For very singular kernels, and correspondingly singular solutions - in general measure solutions - it makes sense to reformulate the problem in Lagrangian coordinates
and work mainly in this framework to develop the theory.  
It is easy to see, in the case of strong solutions, that the above formulation is equivalent to 
 \begin{gather}
 \rho(t)= \sigma^t \# \rho_{init} \label{sol1a},\\ \text{ $\sigma$ is the flow map associated to  the field $v=-(\nabla K * \rho(t))(x)$} \label{sol2a}.
\end{gather}
In other words, the mass $\rho$ is transported by characteristics $\sigma$ that satisfy the ordinary differential equation

$$\frac{d}{dt}{\sigma(x,t)} = v(\sigma(x,t), t) \quad , \quad \sigma(x,0)=x.$$  The map $\sigma^t: \real^d \to \real^d$ is defined by $\sigma^t(x)=\sigma(x,t)$   and  $\sigma^t \# \rho_{init}$ stands for the push forward of the measure $\rho_{init}$ by the map $\sigma^t$ (see below for a precise definition of the push forward).  
We work with formulation (\ref{sol1a}-\ref{sol2a}) to prove existence of solutions, rather than (\ref{agg1}-\ref{agg2}).  We refer to this
as the {\em Lagrangian formulation} of the problem.
Note, in particular, that the flux $\rho v$ in (\ref{agg1}) may be difficult to define,
for the product of a measure $\rho$ and a velocity field that blows up precisely at the point where $\rho$ concentrates mass.   This phenomena can occur, for example, in the case of power law potentials potentials $K(x)=\abs{x}^\alpha / \alpha $, $\alpha<1$, for which existence theory was not known prior to this work. Since we are working with a purely transport problem it is very natural to work in a Lagrangian framework.  The radial symmetry combined with monotonicity provides a focusing effect in which 
the only mass concentration occurs precisely at the origin, providing a natural way to keep track of mass transport in this problem.
 We now introduce some technical notation and corresponding well-known functional analytic results.
\begin{itemize}
\item $\mes(\real^d)$ stands for the space of Borel non-negative measure on $\real^d$ which have finite mass.

\item $\mes_R(\real^d)$ is the set of $\mu \in  \mes(\real^d)$ which are radially symmetric.

\item  $\mes_{RD}(\real^d)$  is the set of $\mu \in  \mes(\real^d)$ which are radially symmetric and decreasing. To be more precise,  $\mu$ belongs to $\mes_{RD}(\real^d)$ if and only if it can be written 
$
\mu=m \delta + g
$,
where  $m\in [0,+\infty)$, $\delta$ is the Dirac delta measure at the origin and $g$ is an $L^1$ function  which is nonnegative,  radially symmetric and monotone decreasing as a function of the radius.

 \item $\prob(\real^d)$, $\prob_R(\real^d)$ and $\prob_{RD}(\real^d)$ are the subset of $\mes(\real^d)$, $\mes_R(\real^d)$ and $\mes_{RD}(\real^d)$ respectively which are made of measure of mass $1$.

\item $\prob_2(\real^d)\subset \prob(\real^d)$, is the subspace of probability measure of finite second moment, i.e. $\int_{\real^d} |x|^2 d\mu(x) <\infty$.

\item {We say that a sequence $(\mu_n) \subset \prob(\real^d)$ converges narrowly to $\mu\in \prob(\real^d)$, denoted by $\mu_n \rightharpoonup \mu$, if
$$\lim_{n\to\infty} \int_{\real^d} f(x) d\mu_n(x) = \int_{\real^d} f(x)  d\mu(x)$$
for every $f \in C_b^0(\real^d)$, the space of continuous and bounded  real function defined on $\real^d$.}

\item $C_{w}([0,+\infty),\prob(\real^d))$ is the set of functions $\mu:[0,+\infty)\to \prob(\real^d)$ which are narrowly continuous, i.e. $\mu(t+h) \rightharpoonup \mu(t)$ as $h\to 0$ $\forall t\geq 0$.

\item For $\mu$ and $\nu$ in $\prob_2(\real^d)$, $W_2(\mu,\nu)$ stands for the Wasserstein distance with quadratic cost between $\mu$ and $\nu$
(See  [65] for the definition and properties of the Wasserstein distance $W_2(\mu,\nu)$).
Recall that $\prob_2(\real^d)$, endowed with the metric $W_2$ is a complete metric space.  Furthermore, $$\lim_{n\to \infty} W_2(\mu_n,\mu) = 0 \;\; \Rightarrow \;\; \mu_n\rightharpoonup\mu \quad \text{ as } n\to\infty.$$
\item $C([0,+\infty),\prob_2(\real^d))$ is the set of functions from $[0,+\infty)$ to $\prob_2(\real^d)$ which are continuous with respect to $W_2$.
 Note that $$C([0,+\infty),\prob_2(\real^d)) \subset C_w([0,+\infty),\prob(\real^d)).$$
 The space $C([0,+\infty),\prob_2(\real^d))$ is endowed with the distance $$\calW_2(\mu,\nu) = \sup_{t\geq 0} W_2(\mu(t),\nu(t)).$$
 
 \item If $T:\real^d\to\real^d$ is a Borel map, and if $\mu\in \mes(\real^d)$, we denote by $T\# \mu$ the push forward of $\mu$ through $T$, defined by $T\#\mu(B) = \mu(T^{-1} (B))$, $\forall B\in \calB(\real^d)$.
More generally we have 
$$
\int_{\real^d} f(T(x)) d \mu(x)= \int_{\real^d} f(x) d (T \# \mu)(x)
$$
for every bounded  Borel function $f : \real^d \to \real$.

 \item Both $B(0,R)$ and $B_R$ will be used to denote the open ball of radius $R$, $\{x\in \real^d: \abs{x} < R \}$. $A_\epsilon$, $\epsilon<1$,  denotes the annulus   $\{ x\in \real^d: \epsilon < |x| < 1\}$.
\end{itemize}
  All the probability measures in this paper are compactly supported  and therefore belong to the space $\prob_2(\real^n)$ on which the Wasserstein distance with quadratic cost is defined.

\subsection{ Lagrangian solutions versus distributional solutions}

\

Let us focus on power law kernels $K=|x|^\alpha/ \alpha$ for simplicity.
 If $v$ is bounded on compact sets, which is only true for $\alpha \ge 1$, it is then standard (see \cite{AGS} or  \cite[Proposition 4.8]{BLR} for example) to prove that if  $\rho$ and $\sigma$ satisfy \eqref{sol1a} and \eqref{sol2a}, then $\rho$ is a distributional solution of the aggregation equation, i.e.
\begin{gather} 
 \int_0^{+\infty} \int_{\real^d}  \Big( \; \frac{d\xi}{dt}(x,t) +
  \nabla \xi(x,t) \cdot v_t (x) \; \Big) \; d\rho_t(x) \; dt =0 \label{dist1},\\
 v_t(x)=v(x,t)= -(\nabla K * \rho(t))(x) \label{dist2}
 \end{gather}
 for all $\xi \in C_0^\infty(\real^d\times (0,+\infty))$. On the other hand, if  $2-d<\alpha < 1$, then the velocity field $x \mapsto v_t(x)$ is not  bounded and it is not clear how to give a sense to \eqref{dist1}.  Hence we use the Lagrangian formulation of the problem throughout most of this paper.  
 In the special case of the Newtonian potential, in Section~\ref{Newtonian} we transform using mass variables to Burgers equation for which it again makes sense to use a distributional form of the problem albeit in a different coordinate system.


\section{General theory of radially symmetric decreasing solutions}
\label{genrad}

\

This section is devoted to the proof of the following theorem:

\begin{theorem} \label{thm:rd}
Let $\nabla K(x)=x {\abs{x}^{\alpha-2}}$, $\alpha \in (2-d,2)$. Given $\rho_{init} \in \prob_{RD}(\real^d)$ with compact support, there exists $\rho \in C([0,+\infty),\prob_{RD}(\real^d))$ and a continuous map $\sigma: [0,+\infty) \times \real^d \to \real^d$ satisfying 
\begin{gather}
 \rho(t)= \sigma^t \# \rho_{init} \label{sol1aa},\\ \text{$\sigma$ is the flow map associated to  the  field }
 v(x,t)=\begin{cases}-(\nabla K * \rho(t))(x),& x\neq 0\\
  0, & x=0.
  \end{cases}
   \label{sol2aa}
\end{gather}

\end{theorem}

\begin{remark}
In \eqref{sol2aa} we  could have let $v(x,t)=-(\nabla K * \rho(t))(x)$ for all $x$ with the understanding $\nabla K(0)=0$ and the convolution is taken in the sense of principle value.
\end{remark}

Theorem \ref{thm:rd} is interesting for two reasons: first it provides global existence of radially symmetric decreasing measure solutions with potential more singular than the one considered previously. In \cite{CDFLS} global existence and uniqueness of measure solutions is proven for $\alpha \ge 1$; here we restrict our attention to radially symmetric decreasing solution but we obtain global existence for $2-d<\alpha<2$. Secondly this theorem shows that radially symmetric decreasing profiles are preserved for all time when $2-d<\alpha<2$. 
Monotonicity is also preserved for the Newtonian case however in this case the problem localizes and the simpler proof is carried out in Section~\ref{Newtonian}. 

\subsection{Formula for the convolution in radial coordinates and properties of the kernel}

\

In this section we recall some known results about radially symmetric solutions of the aggregation equation and we prove additional results  needed in the following subsections.

\begin{definition}\label{definition:radial}
Let  $\mu \in \mes_R(\real^d)$.  We define $\hat{\mu} \in \mes([0,+\infty))$ to be the Borel measure on $[0,+\infty)$ which satisfies
$$
\hat{\mu}(I)= \mu(\{ x\in \real^d:  |x|\in I\})  
$$
for all $I\in {\calB}([0,+\infty))$.
\end{definition}

\begin{remark}
 It is straightforward to check that if a sequence $\mu_n \in \prob_{R}(\real^d)$ converges narrowly to $\mu \in \prob_{R}(\real^d)$ then
\begin{equation} \label{narrow}
\lim_{n\to\infty} \int_{[0,+\infty)} f(r) d\hat{\mu}_n(r) = \int_{[0,+\infty)} f(r)  d\hat{\mu}(r)
\end{equation}
 for every $f \in C_b^0([0,+\infty))$, the space of continuous and bounded  real function defined on $[0,+\infty)$.
\end{remark}

\begin{definition} \label{def:concentration}
Let  $\mu,\nu \in \prob_R(\real^d)$. We say that $\mu$ is more concentrated than $\nu$, and we write $\mu \succ \nu$,  if  ${\hat{\mu}}=T\#{\hat{\nu}}$ for some Borel map $T: [0,+\infty) \mapsto [0,+\infty)$ satisfying
$T(r)\le r$ for all $r \in [0,+\infty)$. 
\end{definition}

\begin{remark} Suppose that   $\mu, \nu \in \prob_{R}(\real^d)$ with $\nu=m \delta + f$ for some $m \ge0$ and  $f \in L^1(\real^d)$. 
 It can then be proven that
\begin{equation} \label{equivalence}
\mu \succ \nu \quad \Longleftrightarrow \quad \mu\left( \overline{B (0,r)}\right) \ge \nu\left(\overline{B(0,r)} \right) \quad \forall r>0.
\end{equation}
This equivalence is actually true for any $\mu,\nu \in \prob_{R}(\real^d)$ if one uses transport plans in Definition \ref{def:concentration} rather than only transport maps
 (see \cite{Vil03} for a definition of  transport plans). 
 The proof is a consequence of the fact that,
  given any two probability measures on $[0,+\infty)$, an optimal transport plan with respect to the quadratic cost can be explicitly constructed in term of the cumulative distributions of these two probability measures \cite[Theorem 2.18 page 74]{Vil03}. It can then be checked that if the cumulative distribution of $\hat{\mu}$ is greater than the cumulative distribution of $\hat{\nu}$,  this optimal transport plan takes elements of mass from $\hat{\nu}$ and move them toward the origin. It can also be checked that if
 $\hat{\nu}=m \delta + \hat{f}$ for some $m \ge0$ and  $\hat{f} \in L^1([0,+\infty))$, then this optimal transport plan is induced by a transport map.
 Since the equivalence  \eqref{equivalence} is not used in this paper, we omit the proof. We note that  other authors (see for example \cite[Chapter 1]{Vazquez}) use the right hand side of \eqref{equivalence}  to define the notion of concentration.
%
\end{remark}

\

For $\alpha \in (2-d,2)$ define the function $\phi: [0,+\infty) \to \real$ by
\begin{equation}
\phi(r)= \frac{1}{\omega_{d-1}}\int_{S^{d-1}} \frac{e_1-r y}{\abs{e_1-r y}^{2-\alpha}} \cdot e_1 \;  d \sigma(y),
\end{equation}
 where $S^{d-1} =\{x\in \real^d: \abs{x}=1\}$ is the unit sphere and $\omega_{d-1}$ its surface measure.
The following lemma was  proven in \cite{BLR} for the case $\alpha=1$ and  in  \cite{HJD} for general $\alpha$.
\begin{lemma} \label{dong}Let $\nabla K(x)=x {\abs{x}^{\alpha-2}}$, $\alpha \in (2-d,2)$. 
Let $\mu \in \mes_{R}(\real^d)$. Then for any $x\neq 0$, we have
\begin{equation} \label{formula}
(\nabla K * \mu ) (x)= \abs{x}^{\alpha-1} \int_0^{+\infty} \phi \p{\frac{r}{\abs{x}}} d \hat{\mu}(r)  \;  \frac{x}{\abs{x}}.
\end{equation}
Moreover, $\phi$ is continuous, strictly positive, non-increasing on $[0,+\infty)$, and
$$
\phi(0)=1, \qquad  \lim_{r \to \infty} \phi(r) r^{2-\alpha}=\frac{d+\alpha-2}{d}.
$$ 
\end{lemma}
 Note  in particular  that $\phi \in C_0^b([0,+\infty))$ which, in view of \eqref{narrow}, will be convenient in order to pass to the limit in expressions such as \eqref{formula}. 
The positivity, monotonicity and boundedness of $\phi$ have three important consequences that can be directly read from \eqref{formula}.
\begin{corollary}
Let $\mu, \nu \in \mes_{R}(\real^d)$.  Since $\phi$ is strictly positive, we have 
\begin{equation} \label{comp1}
\mu \ge \nu  \quad  \Longrightarrow  \quad   \abs{\nabla K * \mu} \ge \abs{\nabla K * \nu } .
\end{equation}
Let  $\mu, \nu \in \prob_R(\real^d)$.  Since $\phi$ is non-increasing, we have
\begin{equation}  \label{comp2}
\mu \succ \nu  \quad  \Longrightarrow  \quad   \abs{\nabla K * \mu} \ge \abs{\nabla K * \nu } .
\end{equation}
Let  $\mu \in \prob_R(\real^d)$.  Since $0 < \phi \le 1$ we have
\begin{equation}  \label{comp33}
   \abs{\nabla K * \mu} \le  \abs{x}^{\alpha-1} .
\end{equation}
\end{corollary}

In the next Lemma we prove that that $\phi$ is 
 $C^1$ for $\alpha> 3-d$, quasi-Lipschitz continuous for $\alpha=3-d$ and H\"older continuous for $2-d <\alpha < 3-d$. The lack of smoothness for  $\alpha \le 3-d$ is due to a singularity in the derivative $r=1$. 
Below we prove sharp estimates on the regularity of $\phi$; later we will only use the fact that $\phi$ is  H\"older continuous in the range of $\alpha$ considered.

\begin{lemma} \label{phi-is-holder} 
\begin{enumerate}
\item[(i)] If $\alpha \in (3-d,2)$ then $\phi\in C^1(0,+\infty)$ and $\phi'$ is bounded on $(0,+\infty)$.
\item[(ii)] If $\alpha=3-d$, then there exists a constant $C_1>0$ such that
\begin{equation} \label{holder}
\abs{\phi(r_1)-\phi(r_2)} \le C_1 \abs{r_1-r_2}(1- \log \abs{r_1-r_2}  )
\end{equation}
for all $r_1, r_2$ satisfying $\abs{r_1-r_2}<1/2$. 
\item[(iii)] If $\alpha \in (2-d, 3-d)$, then  there exists a constant $C_2>0$ such that
\begin{equation} \label{holder}
\abs{\phi(r_1)-\phi(r_2)} \le C_2 \abs{r_1-r_2}^{\alpha-(2-d)} 
\end{equation}
for all $r_1, r_2$ satisfying $\abs{r_1-r_2}<1/2$. 
\end{enumerate}
\end{lemma}

\begin{proof}
 In \cite[Lemma 4.4]{HJD}  it was proven that  $\phi$ is differentiable on $[0,1)\cup (1+\infty)$ and that, for $r \neq 1$:
\begin{gather} \label{gagi}
\phi' (r)= -C_{\alpha,d} \int_0^\pi \frac{r (\sin \theta)^d}{A(r,\theta)^{4-\alpha }}d \theta \\
\text{where } A(r, \theta) = (1 + r^2 -2 r \cos \theta)^{1/2}  \label{ding}  \\
\text{and }   C_{\alpha,d}=\frac{\omega_{d-2}(2- \alpha)(d+\alpha-2)}{\omega_{d-1} (d-1)}>0.
\end{gather}
Note first that for fixed $r$ the function $\theta \mapsto A(r, \theta)$ reaches its minimum at $\theta=0$. So $A(r,\theta) \ge \abs{r-1}$ and one can easily see from \eqref{gagi}   that for all the $\alpha$ considered, $\phi'(r)$ is bounded on $[0,1/2] \cup [3/2,+\infty)$. It is therefore enough to prove the statements of the Lemma only on the interval $(1/2,3/2)$.

We first prove (i).  Note that for fixed  $\theta$ the function $r \mapsto A(r, \theta)$ reaches its minimum at $r=\cos \theta$ and therefore $A(r,\theta) \ge \abs{\sin \theta} $. As a concequence we have the following estimate for the integrand in \eqref{gagi}:
\begin{equation}
 \frac{r (\sin \theta)^d}{A(r,\theta)^{4-\alpha }} \le \frac{3/2}{(\sin \theta)^{4-\alpha-d}} \qquad \text{ for all $r \in \left(\frac{1}{2},\frac{3}{2}\right)$ and $\theta \in [0,\pi]$.}\label{john}
\end{equation}
Since the right hand side of (\ref{john}) 
is integrable if $\alpha>3-d$ we obtain (i) by the dominated convergence theorem.

We now turn to the proof of (ii) and (iii).  We first derive the estimates
\begin{align}
\abs{\phi'(1+h)} \le  \frac{C}{\abs{h}^{3-\alpha-d}} &   \qquad \text{ if } \alpha \in (2-d,3-d) 
  \label{vovo}\\
\abs{\phi'(1+h)} \le - C \log \abs{h}  &  \qquad \text{ if } \alpha= 3-d \label{wowo}
\end{align}
for all $h \in (-1/2,1/2)$, $h \neq 0$. The constant $C>0$ depends on $\alpha$ but not on $h$.
We will  prove \eqref{vovo}  only for $h \in (0,1/2)$. The proof for  $h \in (-1/2,0)$ is precisely the same.  Write
\begin{multline*}
\abs{\phi' (1+h)}
= C_{\alpha,d} \left( \int_h^{\pi-h}\frac{ (1+h) (\sin \theta)^d}{A(1+h,\theta)^{4-\alpha }} d \theta   +
 \int_{[0,h] \cup [\pi-h, \pi] }\frac{ (1+h) (\sin \theta)^d}{A(1+h,\theta)^{4-\alpha }} d \theta \right)
\end{multline*}
and let $(I)$ be the first integral and $(II)$ the second one. Using the fact that $A(r,\theta) \ge \sin \theta$ we find that
\begin{multline*}
 (I) \le \frac{3}{2} \int_h^{\pi-h}\frac{ 1}{(\sin \theta)^{4-\alpha-d}}  d \theta 
 \le 3  \int_h^{\pi/2}\frac{ 1}{(\sin \theta)^{4-\alpha-d}}  d \theta \\
 \le 3  \int_h^{\pi/2}\frac{ 1}{( \frac{2}{\pi} \theta)^{4-\alpha-d}}  d \theta 
 \le \frac{C}{h^{3-\alpha-d}}
\end{multline*}
where we have used the symmetry of $\sin \theta$ around $\theta=\pi/2$ and the fact that $\sin \theta \ge (2/ \pi) \theta$ on the interval $[0, \pi/2]$. To estimate $(II)$ we use  the fact that $A(1+h, \theta) \ge h$:
\begin{multline*}
(II) \le \frac{3/2}{h^{4-\alpha}} \int_{[0,h] \cup [\pi-h, \pi] }  (\sin \theta)^d d\theta \le \frac{3}{h^{4-\alpha}} \int_0^h (\sin \theta)^d d\theta \\
\le \frac{3}{h^{4-\alpha}} \int_0^h  \theta^d d\theta   \le \frac{C}{h^{3-\alpha-d}}.
\end{multline*}
This concludes the proof of \eqref{vovo}. The proof of \eqref{wowo} is similar. Let 
\begin{equation}
\omega(r)= \begin{cases}
\frac{C}{\alpha-(2-d)}r^{\alpha-(2-d)} &  \text{ if }\alpha \in (2-d,3-d)\\
C r (1-\log r) &  \text{ if }\alpha =3-d
\end{cases}
\end{equation}
and note that $\omega$ is the antiderivative of the right hand side of \eqref{vovo} and \eqref{wowo}. Note also that $\omega$ is a nonnegative, increasing, concave function on $[0,1]$ which is equal to $0$ at $r=0$. 
To conclude the proof  of (ii) and (iii) we need to show that $\omega$ is the modulus of continuity of $\phi$ on the interval $(1/2,3/2)$, that is 
 \begin{equation} \label{modulus}
 \abs{\phi (r_1) -\phi(r_2)} \le \omega(\abs{r_1-r_2}) \qquad \text{for all $r_1,r_2 \in \left(\frac{1}{2},\frac{3}{2}\right) $}. 
 \end{equation}
Since $\phi'$ is negative, from  \eqref{vovo} and \eqref{wowo} we have that 
$
0\le -\phi'(1+h) \le \omega'(h). 
$ for $h>0$.
Integrating this inequality on $[h_1,h_2]$  and using the fact that $\omega(h_2)-\omega(h_1)\le \omega(h_2-h_1)-\omega(0)= \omega(h_2-h_1)$ due to the concavity of $\omega$, we obtain that
\begin{equation*} \label{vish1}
0 \le \phi(1+h_1)-\phi(1+h_2) \le \omega(h_2-h_1) \qquad  \text{ for all } 0 \le h_1 < h_2 < 1/2.
\end{equation*}
This prove that \eqref{modulus} holds for all  $r_1, r_2 \in [1,3/2)$. A similar proof leads to the same result on the interval
$(1/2,1]$. To obtain the result on the full interval $(1/2,3/2)$, let $h_1,h_2 \in (0,1/2)$ and write 
\begin{multline*}
0 \le \phi(1-h_1)-\phi(1+h_2)=  \phi(1-h_1)-\phi(1)+\phi(1)-\phi(1+h_2) \\
 \le \omega(h_1)+ \omega(h_2) \label{tata2}   \le  2 \omega(h_1+h_2).
\end{multline*}
 To obtain the last inequality
we have used the fact that $\omega$ is increasing.
  \end{proof}
\subsection{Regularity of the velocity field}

\

We now study the regularity of a the velocity field associated with a radially symmetric decreasing measure solution of the aggregation equation.
\noindent Recall that
\begin{equation}
A_\epsilon:= \{ x\in \real^d: \epsilon < |x| < 1\}.
\end{equation}  
Obviously $A_0= B(0,1)\bs \{0\}$.
\begin{proposition}  \label{regularity} Let  $\rho\in C_w([0,+\infty), \prob_{RD}(\real^d))$ and assume that $\text{supp}(\rho(t)) \subset B(0,1)$ for all $t\ge 0$. Let $\nabla K(x)=x {\abs{x}^{\alpha-2}}$, $\alpha \in (2-d,2)$.
 Then 
 the velocity field $v(x,t)$ defined by
 \begin{equation}
 v(x,t)=\begin{cases}-(\nabla K * \rho(t))(x),& x\neq 0\\
  0, & x=0
 \end{cases}
 \end{equation}
satisfies:
 \begin{enumerate}
\item[(P0)] $v(x,t)$ is continuous on $A_0 \times [0,+\infty)$ for all $t \ge 0$.
\item[(P1)] For every $t \ge 0$, the function $x\mapsto  v(x,t)$ is continuously differentiable on $A_0$. 
\item[(P2)] Given $\epsilon>0$ there exists $C>0$ such that $\abs{\nabla v(x,t)}< C$ for all $(x,t) \in A_\epsilon \times [0,+\infty)$.
\item[(P3)] Given    $\epsilon>0$ and $\eta>0$,  there exists $\delta>0$ such that
$$   \abs{\nabla v(y,t)-\nabla v(x,t)} \le \eta $$
for all $x,y \in A_\epsilon$ satifying $\abs{x-y}< \delta$ and for all $t \ge 0$.
\end{enumerate} 
\end{proposition}

\

 The notation $\nabla v(x,t)$ stands for the derivative of $v$ with respect to $x$.
In order to prove this proposition,  we will need the following lemma and its corollary.  
\begin{lemma} \label{keylemma}
Suppose $g\in L^1(\real^d)$ is nonnegative, radially symmetric decreasing, and supported in $B(0,1)$.  Suppose also that  $\Phi \in C(\real^d \backslash  \{0\}) \cap L^1_{loc}(\real^d)$. Then $\Phi*g \in C(\real^d \backslash  \{0\}) \cap  L^1_{loc}(\real^d)$ and 
\begin{equation}
\label{bibi}
\|\Phi*g\|_{L^\infty(A_{\epsilon})}  \le \|g\|_{L^1(B_{\epsilon/2})}  \Big\{
\sup_{\epsilon/2<|y|<2}|\Phi(y)|+ \frac{\|\Phi\|_{L^1(B_2)}}{|B_{\epsilon/2}|}
\Big\}.
\end{equation}
\end{lemma}
\begin{proof}
 Since
 $\Phi$ belongs to $L^1_{loc}$, by Young's inequality  $\Phi*g$ is also in $L^1_{loc}$. We now prove estimate \eqref{bibi}.
Fix $x$ in $A_\epsilon$.
\begin{align}
\abs{(\Phi*g)(x)} &\le \int_{B_{\epsilon/2}} \abs{\Phi(x-y)}   g(y) dy + \int _{A_{\epsilon/2} } \abs{\Phi(x-y)}  g(y) dy\\
\label{xavier}
& \le     \Big( \sup_{\epsilon/2<|y|<1+\epsilon/2}|\Phi(y)|\Big) \; \|g\|_{L^1(B_{\epsilon/2})} 
\\ & \hspace{5cm}+ \Big( \sup_{y \in A_{\epsilon/2}} g(y) \Big) \|\Phi\|_{L^1(B_2)} \label{xavier2}.
\end{align}
Since $g$ is radially symmetric decreasing, we have that $g(z) \ge  \sup_{y \in A_{\epsilon/2}} g(y)$ for almost every $z\in B_{\epsilon/2}$. Therefore we obtain
$$
\| g\|_{L^1(B_{\epsilon/2})} \ge  |B_{\epsilon/2} | \;\;\sup_{y \in A_{\epsilon/2}} g(y)
$$
which, combined with \eqref{xavier2}, leads to the desired estimate. We now prove that $\Phi*g$ is continuous.  Reasoning as above we obtain that if $x\in A_\epsilon$ then
 \begin{multline} \label{coco}
\abs{(\Phi*g)(x+h)-(\Phi*g)(x)}   \le   \|g\|_{L^1(B_{\epsilon/2})}  \\   \Big\{
\sup_{\epsilon/2<|y|<2}|\Phi(y+h)-\Phi(y)|+ \frac{\|\Phi(\cdot +h ) - \Phi(\cdot) \|_{L^1(B_2)}}{|B_{\epsilon/2}|}
\Big\}.
\end{multline}
We conclude that $\Phi* g$ is continuous using
 the fact that $\Phi$ is uniformly continuous on compact sets which do not contain the origin and the continuity of the translation $h \to  \Phi(\cdot+h)$ from $\real^d$ to $L^1(B_2)$.
\end{proof}

Recall that $\prob_{RD}(\real^d)$ is the space of probability measure $\mu$ which can be written
$$
\mu = m \delta + g
$$
for some $m \ge 0$ and for some nonnegative,  radially symmetric decreasing function $g\in L^1(\real^d)$.  From the previous Lemma we directly obtain the following corollary: 
\begin{corollary} \label{cor:key-estimate}
Let $K: \real^d \to \real$ be a potential such that $K_{x_i x_j}\in C(\real^d \backslash  \{0\}) \cap L^1_{loc}(\real^d)$. Then the family of functions $\{K_{x_i x_j}*\mu : {\mu \in \prob_{RD}(\real^d)}\}$ is uniformly bounded and equicontinuous on every annulus  $A_\epsilon$, $\epsilon>0$. To be more precise we have: 
 \begin{equation}
\label{bibibi}
\|K_{x_i x_j}*\mu\|_{L^\infty(A_{\epsilon})} \le 
\sup_{\epsilon/2<|y|<2}|K_{x_i x_j}(y)|+ \frac{\|K_{x_i x_j}\|_{L^1(B_2)}}{|B_{\epsilon/2}|}
 \qquad \text{ for all $\mu \in  \prob_{RD}(\real^d)$ }
\end{equation}
And also: given $\epsilon>0$ and $\eta>0$, there exists $\delta>0$ such that 
\begin{equation} \label{bibibi2}
 \abs{(K_{x_i x_j}*\mu) (x) -  (K_{x_i x_j}*\mu) (y)} \le \eta 
 \end{equation}
for all $x,y \in A_\epsilon$ satisfying $\abs{x-y}< \delta$ and for all $\mu \in  \prob_{RD}(\real^d)$.
\end{corollary}
\begin{proof}
Since $\mu \in  \prob_{RD}(\real^d)$ it can be written $\mu=m \delta +g$ where $m \in [0,1]$ and $g$ satisfies the hypothesis of  Lemma \ref{keylemma}. So $K_{x_i x_j}*\mu= m K_{x_i x_j} +K_{x_i x_j}*g$ and it is easy to conclude using   Lemma \ref{keylemma} with $\Phi=K_{x_ix_j}$.  The second statement is a consequence of \eqref{coco}. 
\end{proof}

We now prove Proposition  \ref{regularity}:

{\it Proof of Proposition \ref{regularity}.} Let us first check that $v(x,t)$ is continuous on $A_0 \times [0,+\infty)$. Continuity with respect to time comes from the fact that $t \mapsto \rho(t)$ is narrowly continuous together with the fact that the kernel $\phi$ appearing in formula \eqref{formula} belongs to $C^0_b([0,+\infty))$.  Continuity with respect to space comes from the H\"older continuity of $\phi$. 
 To prove (P1), (P2) and (P3), note first that if $K(x)=\abs{x}^{\alpha}$, $\alpha>2-\alpha$, then $K_{x_i x_j}\in C(\real^d \backslash  \{0\}) \cap L^1_{loc}(\real^d)$,  and then use Corollary \ref{cor:key-estimate}.  \qed 

\

\begin{definition}
$\vel$ is the space of velocity fields $v: B(0,1) \times [0,+\infty) \to \real^d$ which are radially symmetric and pointing inward (i.e. $v(0,t)=0$ and $v(x,t)=- \lambda (\abs{x},t) {x}$, $|x|>0$,  for some nonnegative function $\lambda: (0,1) \times [0,+\infty) \to \real$) and which satisfies (P0)--(P3). 
\end{definition}

Obviously a velocity field defined as in the statement of Proposition \ref{regularity} belongs to $\vel$ (the fact that it points inward comes from formula \eqref{formula} together with the positivity of the kernel $\phi$).
We now investigate properties of  flow maps generated by  velocity fields in $\vel$.
\begin{proposition}
 \label{focussing}
Suppose $v \in \vel$. Then there exists
a unique continuous function $\sigma: B(0,1) \times [0,+\infty) \mapsto B(0,1)$ satisfying
 \begin{equation}\label{ode33}
 \sigma(x,t)=x + \int_{0}^t v(\sigma(x,s),s) ds \qquad\text{ for all }(x,t) \in B(0,1) \times [0,+\infty).
\end{equation}
 Moreover if the point $(x_0,t_0) \in  B(0,1) \times[0,+\infty)$ is such that $\sigma(x_0,t_0) \neq 0$, then the mapping $ x \mapsto \sigma(x,t_0)$ is continuously differentiable at $x_0$ and we have
 \begin{equation}\label{ode555}
 \text{det }\nabla\sigma(x_0,t_0) =   \text{exp} \big( {\int_{0}^{t_0} (\text{div }v)( \sigma(x_0,s),s) ds} \big) >0 .
\end{equation}
In particular, if $R(t)$ is such that $(\sigma^t)^{-1}(\{0\})= \overline{B(0,R(t))}$, then $\sigma^t$ is a diffeomorphism from $B(0,1) \bs \overline{B(0,R(t))}$ to $\real^d \bs \{0\}$. 
\end{proposition}
\begin{proof} The global existence and forward uniqueness of solution of the ODE $\dot{x}=v(x,t)$ simply come from the fact that  $v$ is Lipschitz continuous with respect to space away from the origin (because of (P2)) together with the fact that $v$ is pointing inward. 
Indeed, solutions  can be continued as long as they are in $B(0,1)\bs \{0\}$ and since  $v$ is pointing inward, the only way for a solution to escape this domain is to reach the origin. When a solution  reaches the origin, it stays there forever in accordance to the fact that $v(0,t)=0$ for all $t \ge 0$.

The differentiability of $\sigma^t$ on $B(0,1) \bs \overline{B(0,R(t))}$ and formula \eqref{ode555} are more delicate. In classical ODE textbooks, such results are obtained under the assumption that $v(x,t)$ is continuously differentiable  in both space and time. In our case  $v$ is continuously differentiable in space but only continuous in time. However, by revisiting classical proofs, one can easily check that assumption (P3) is enough to obtained differentiability of the flow map as well as formula \eqref{ode555}. This is done in section \ref{ode-appendix} of the appendix.
 \end{proof}


Since the flow map $\sigma^t$ generated by a velocity field in $\vel$ is a diffeomorphism from $B(0,1) \bs \overline{B(0,R(t))}$ to $B(0,1) \bs \{0\}$, we can use the change of variable formula to express the push forward of a measure in $\prob_{RD}(\real^d)$ by $\sigma^t$:

 \begin{corollary}
 \label{push-forward}
 Suppose $v \in \vel$ and let $\sigma$ be the associated flow map provided by Proposition \ref{focussing}.
Let $\mu=m_0 \delta + g_0 \in \prob_{RD}(\real^d)$ and assume that $\text{supp} \mu \in B(0,1)$. Then
$$
\sigma^t \# \mu = m(t) \delta + g(t) 
$$
 where $m(t) \in \real^+$ and $g(t) \in L^1(\real^d)$  satisfy :
\begin{gather} 
m(t)=m_0+\int_ {(\sigma^t)^{-1}(\{0\}) }g_0(x) \; dx  \label{bobo}\\
g(x,t) 
=  \p{\frac{g_{0}}{\text{det} \, \nabla \sigma^t} \circ (\sigma^t)^{-1}}(x) \quad \text{ for }x \neq 0\label{ohi}.
\end{gather}
\end{corollary}

\subsection{Radially symmetric decreasing  profiles are preserved}

\

In this subsection we give a heuristic argument (which will be made rigorous in the next subsection) explaining why
radially symmetric decreasing profiles are preserved by the aggregation equation when $2-d<\alpha\le 2$. We recall here that it is observed numerically that when $\alpha>2$, radially symmetric decreasing profiles are not preserved.
A key ingredient in our argument is the known fact that the convolution of two radially symmetric decreasing functions is still radially symmetric decreasing (see \cite{Lieb83} for example). For completeness we give a quick proof of this fact:
\begin{lemma} \label{layercake}  Suppose  $g\in L^1(\real^d)$ with compact support and  $f\in L^1_{loc}(\real^d)$.  If $f$ and $g$ are nonnegative radially symmetric decreasing  functions, then $f*g$ is also a nonnegative radially symmetric decreasing function.
\end{lemma} 
\begin{proof}
The mononicity of $g$ allows us to use a ``layer cake" decomposition of g, namely
 $$g(x) =\int_0^\infty\chi_{B(0,\tilde r(g))} (x) dg$$
 where $\tilde r(g)$ denotes the inverse function of $g(r)$ and $\chi_{B(0,s)}$ denotes the characteristic function of the ball of radius $s$.
 Thus 
 \begin{equation}
 f*g(x)  = \int_0^\infty f*\chi_{B(0,\tilde r(g))} (x) dg
 \label{conv}\end{equation}
 and we note that  the integrand of (\ref{conv}) is monotone decreasing
 because the characteristic function of a ball convolved with a nonnegative $L^1_{loc}$ monotone decreasing function is itself  monotone decreasing.  By integrating a monotone integrand with respect to g we obtain 
the monotonicity result for $f*g$.
\end{proof}

\

We now present the heuristic argument. Let us assume that $u(x,t)$ is a smooth solution of the aggregation equation. Clearly we have:
\begin{equation} \label{pde111}
\frac{\partial{u}}{{\partial t}}  + \nabla u \cdot {v}= (\Delta K * u) u. 
\end{equation}
Suppose that 
 $\Delta K$ is locally integrable, nonnegative  and radially symmetric decreasing.
When  $\nabla K(x)=x |x|^{\alpha-2}$, these hold true if and only if  $2-d<\alpha\le 2$.
We then use  Lemma \ref{layercake} to see that  if for some $t_0\ge 0$,  $u(\cdot,t_0)$ is radially symmetric decreasing  then the right hand side of \eqref{pde111} is also   radially symmetric decreasing at $t_0$.  This indicates that the rate of change along the characteristic is greater the closer we are to the origin.
Therefore the solution  is expected to remain radially symmetric decreasing for $t>t_0$.
For the special case of the Newtonian potential, $\Delta K*u=u$ and monotonicity is similarly preserved - this is discussed in more detail in Section~\ref{Newtonian}. 

\subsection{Proof of Theorem \ref{thm:rd}}

\

The  proof is inspired by the work in \cite{Canizo}, where global existence of measure solutions for some kinetics model was obtained by using a fixed point iteration in the space of probability measures endowed with the Wasserstein distance. 

Let $\nabla K(x)=x {\abs{x}^{\alpha-2}}$, $\alpha \in (2-d,2)$, and  let $\rho_{init} \in \prob_{RD}(\real^d)$ with $\text{supp}(\rho_{init}) \subset B(0,1)$. Define:
\begin{align*}
&\rho_0(t)=\rho_{init} \qquad \qquad \qquad \forall t \in [0,+\infty)  \\
& 
 v_0(x,t)=\begin{cases}-(\nabla K * \rho_0(t))(x),& \text{ if } x\neq 0\\
  0, & \text{ if } x=0
 \end{cases} 
  \qquad \forall t \in [0,+\infty) \\
&\sigma_0^t : \real^d \to \real^d := \text{flow map associated with } v_0&
\end{align*}
and for $n \ge 1$ define recursively 
\begin{align*}
&\rho_n(t)= \sigma^t_{n-1} \# \rho_{init} \qquad\qquad  \forall t \in [0,+\infty)  \\
& 
 v_n(x,t)=\begin{cases}-(\nabla K * \rho_n(t))(x),& \text{ if } x\neq 0\\
  0, & \text{ if } x=0
 \end{cases} 
  \qquad \forall t \in [0,+\infty) \\
&\sigma_n^t : \real^d \to \real^d := \text{flow map associated with } v_n.
\end{align*}

\begin{proposition}  \label{prop:it} 
\
\begin{enumerate}
\item[(i)]  For all $n\ge 0$, $\rho_n \in C([0,+\infty), \prob_{RD}(\real^d)) $  and $\text{supp}(\rho_n(t)) \subset B(0,1)$ for all $t\ge 0$.
\item[(ii)] 
Given $\epsilon>0$, there exists  $L_\epsilon > 0$  such that
$$
\abs{v_n(x,t)-v_n(y,t)} \le L_\epsilon \abs{x-y}
$$
for all $x,y \in A_\epsilon$, for all  $t \ge 0$, and for all $n \ge 0$. 

\item[(iii)] There exists a constant $\theta \in (0,1]$  depending only on $\alpha$ such that the following holds:  
 Given $\epsilon>0$, there exists $C_\epsilon>0$ and $\delta>0$ such that $\abs{t-s} < \delta$ implies
$$
\abs{v_n(x,t)-v_n(x,s)} \le C_\epsilon \abs{s-t}^{\theta}
$$
for all $x \in A_\epsilon$ and for all $n \ge 0$. 
\item[(iv)] $\rho_{n+1}(t) \succ \rho_n(t)$ for all $n\ge 0$ and all $t\in[0,+\infty)$. This implies  $|v_{n+1}(x,t)| \ge |v_n(x,t)|$ for all $(x,t) \in \real^d \times [0,+\infty)$ and for all $n\ge 0$. 
\end{enumerate}
\end{proposition}

\

Before we prove  this proposition let us explain how it will be used in the proof of  Theorem \ref{thm:rd}.  Because of statements (ii), (iii), (iv) and  the bound $\abs{v_n(x,t)}\le \abs{x}^{\alpha-1}$ (see \eqref{comp33}), we can use the Arzela-Ascoli theorem to conclude   that the $v_n$'s converge uniformly on $A_\epsilon \times [0,+\infty)$ to some function $v$  which is Lipschitz continuous in space and H\"older continuous in time, with same constants $L_\epsilon$ and $C_\epsilon$. Since $\epsilon$ can be chosen as small as we want, $v(x,t)$ is well define on $B(0,1)\bs \{0\} \times [0,+\infty)$. Let $v(0,t)=0$ so that $v$ is now well defined on $B(0,1) \times [0,+\infty)$. This velocity field $v(x,t)$ generates a flow map $\sigma^t: B(0,1) \to B(0,1)$
 and from this flow map we can construct $\rho(t)=\sigma^t\# \rho_{init}$.  In Proposition \ref{prop:unif-conv} it will be shown that $\sigma_n$ converges uniformly to $\sigma$ on $B(0,1) \times [0,\infty)$. This implies in particular that for a given $t$, $\rho_n(t)$ converges narrowly to $\rho(t)$. The narrow convergence preserves the monotonicity (see Proposition \ref{preserve} of the Appendix), and therefore $\rho(t)$ is radially symmetric 
 decreasing. In order to prove  that the radially symmetric 
decreasing function $\rho(t)$ and the flow map $\sigma^t$ obtained by the above limiting process satisfy \eqref{sol1aa} and \eqref{sol2aa} we just need to show that $v(x,t)=-(\nabla K *\rho(t))(x)$ for $x\neq 0$, and this fact will follow  easily from passing to the limit in the relationship $v_n(x,t)=-(\nabla K *\rho_n(t))(x)$ for $x \neq 0$.

{ \it Proof of Proposition \ref{prop:it}.}
Let us first prove (i). The initial iterate $\rho_0(t) \equiv \rho_{init}$ obviously belongs to $C([0,+\infty), \prob_{RD}(\real^d))$ with $\text{supp}(\rho_0(t)) \subset B(0,1)$ for all $t\ge 0$.
Assume that  $\rho_n \in C([0,+\infty), \prob_{RD}(\real^d))$ with $\text{supp}(\rho_n(t)) \subset B(0,1)$ for all $t\ge 0$. From Proposition \ref{regularity} $v_n \in \vel$ and from Corollary \ref{push-forward}
\begin{align} \label{oh}
&\rho_{n+1}(t)  = m_{n+1}(t) \delta + g_{n+1}(t) \\
&m_{n+1}(t)=m_0+\int_ {(\sigma_n^t)^{-1}(\{0\}) }g_0(x) \; dx \\
& \label{baba}g_{n+1}(x,t) 
=  \p{\frac{g_{0}}{\text{det} \, \nabla \sigma_n^t} \circ (\sigma_n^t)^{-1}}(x) \quad \text{ for }x \neq 0.
\end{align}
Here $m_0$ and $g_0$ are such that $\rho_{init}=m_0 \delta + g_0$.
Also from Proposition \ref{focussing} we know that $\text{det} \nabla \sigma_n^t$ satisfies
\begin{equation} \label{bibifoc}
\text{det} \nabla\sigma_n^t(x) = \text{exp} {\int_0^t (\text{div } {v_n}) (\sigma_n^s(x),s) ds }
\end{equation}
for all $(x,t)$ such that $\sigma_n^t(x) \neq 0$.
Since we have assumed that $\rho_n(t)$ is in $\prob_{RD}(\real^d)$ with compact support, and since for $\alpha \in (2-d,2)$ $\Delta K$ is nonnegative, radially symmetric decreasing and locally integrable, we know from Lemma \ref{layercake} that the function 
$
x \mapsto - \text{div} v_n(x,t) = [\Delta K * \rho_n(t)](x)
$
is nonnegative,  radially symmetric and decreasing. 
Since $|x|\le |y|$ implies  $|\sigma_n^s(x)| \le |\sigma_n^s(y)|$  one can easily see from \eqref{bibifoc} that
$$
x\mapsto \frac{1}{\text{det} \nabla \sigma_n^t(x)} \text{ is nonnegative,  radially symmetric and decreasing.}
$$
Then we easily see from \eqref{baba} that, since $g_0$ is radially symmetric and decreasing, so is $x \mapsto g_{n+1}(x,t)$.

Let us now remark  that the estimate $\abs{v_n(x,t)}\le \abs{x}^{\alpha-1}$ together with Lemma \ref{holderlem} of the Appendix lead to the following: if  $\alpha \in (2-d,1)$ then
\begin{equation} \label{holderrr}
\abs{\sigma_n^t(x)-\sigma_n^s(x)} \le C_\alpha  \abs{t-s}^{\frac{1}{2-\alpha}}
\end{equation}
for all $x\in B(0,1)$ and for all $t,s \ge0$. Here $C_\alpha:= (2-\alpha)^{\frac{1}{2-\alpha}}$. If $\alpha \in [1,2)$ then $v(x,t)\le 1$ on
$B(0,1) \times [0,+\infty)$ and therefore we get
\begin{equation} \label{qwerty}
\abs{\sigma_n^t(x)-\sigma_n^s(x)} \le  \abs{t-s}
\end{equation} for all $x\in B(0,1)$ and for all $t,s \ge0$.
 Using Lemma~\ref{canizo} from the Appendix, together with \eqref{holderrr}
   we  obtain that,  if $\alpha \in (2-d,1)$ then
 $$W_2(\rho_{n+1}(t),\rho_{n+1}(s))\leq \|\sigma_n^t-\sigma_n^s\|_{L^\infty(B(0,1))}\leq C_\alpha |t-s|^{\frac{1}{2-\alpha}}.$$
 This prove that $t \mapsto \rho_{n+1}(t)$ is H\"older continuous with respect to $W_2$ when $\alpha \in (2-d,1)$. If $\alpha \in [1,2)$, we obtain from \eqref{qwerty} that $t \mapsto \rho_{n+1}(t)$ is Lipschitz continuous with respect to $W_2$.
 
 \
 
 Statement (ii) is a direct consequence of 
 Corollary \ref{cor:key-estimate}.   We now prove  (iii). Suppose $2-d< \alpha <1$.  Recall from Lemma \ref{phi-is-holder} that $\phi$ is  $\gamma$-H\"older continuous  for some $\gamma \in (0,1]$. 
Choose $\delta$ such that $\abs{t-s}< \delta$ implies $C_\alpha  \abs{t-s}^{\frac{1}{2-\alpha}} \le \epsilon/2$.
 Using Lemma \ref{phi-is-holder} and  estimate \eqref{holderrr} we obtain that $\abs{t-s}< \delta$ and $x \in A_\epsilon$ implies that
\begin{align}
\abs{v_n(x,t)-v_n(x,s)} &=  \abs{x}^{\alpha-1} \int_0^{+\infty} \abs{ \phi\p{\frac{\sigma^t_{n-1}(r)}{\abs{x}}}-\phi\p{\frac{\sigma^s_{n-1}(r)}{\abs{x}}}} d \hat{\rho}_{init}(r) \label{a1a}\\
& \le   \abs{x}^{\alpha-1-\gamma} \int_0^{+\infty} c\abs{ \sigma^t_{n-1}(r)-\sigma^s_{n-1}(r)}^\gamma d \hat{\rho}_{init}(r) \label{b2b}\\
& \le  c  \;C_\alpha  \abs{x}^{\alpha-1-\gamma} \abs{t-s}^{\frac{\gamma}{2-\alpha}} \label{c3c}.
\end{align}
The first equality is a simple consequence of formula  \eqref{formula}, the fact that $\rho_n(t)= \sigma^t_{n-1}\# \rho_{init}$ and the definition of the push forward.  
Note that $C_\alpha  \abs{t-s}^{\frac{1}{2-\alpha}} \le \epsilon/2$ and  \eqref{holderrr}  imply that  $\abs{ \frac{\sigma^t_{n-1}(r)}{\abs{x}}-\frac{\sigma^s_{n-1}(r)}{\abs{x}}} \le 1/2$ for $x \in A_\epsilon$. This allowed us to use  Lemma \ref{phi-is-holder} in order to go from \eqref{a1a} to \eqref{b2b}. The case $\alpha \in [1,2)$ is dealt with similarly.

 \

We finally prove (iv). Obviously  $\rho_1(t) \succ \rho_0(t)\equiv \rho_{init}$ for all $t\ge0$. Assume that for a given $n$,  $\rho_n(t) \succ \rho_{n-1}(t)$ for all $t \ge 0$. Then  \eqref{comp2} implies $|v_n(x,t)|\ge |v_{n-1}(x,t)|$. Lemma \ref{vel-comp}, which is proven in the next section, implies then that $\rho_{n+1}(t) \succ \rho_{n}(t)$ for all $t \ge 0$. 
\qed

 As already mentioned, (ii) (iii) and (iv) imply that the sequence $\{v_n\}$ converges uniformly on $A_\epsilon \times [0,+\infty)$ to some function $v$ (which is Lipschitz continuous in space away from the origin). Setting $v(0,t)=0$ we obtain a velocity field well defined on $B(0,1) \times [0,+\infty)$. This velocity field $v(x,t)$ generates a flow map $\sigma^t: B(0,1) \to B(0,1)$. 

\vspace{.4cm}

\begin{proposition} \label{prop:unif-conv}
 $\sigma_n(x,t)$ converges uniformly to $\sigma(x,t)$ on $B(0,1) \times [0,+\infty)$.
\end{proposition}

\begin{proof}
Let $\epsilon>0$ be fixed. From formula \eqref{formula} it is clear that $\abs{v_0}$ is strictly positive away from the origin. Since $\abs{v_{n+1}} \ge \abs{v_n}$ we have that $\abs{v}$ is also strictly positive away from the origin. Therefore there exists a time
 $T_\epsilon>0$   such that $\sigma^{T_\epsilon}(B(0,1)) \subset B(0,\epsilon)$.
 Choose $N$ so that  $n \ge N$ implies $\norm{v-v_n}_{L^\infty(A_\epsilon \times [0,T_\epsilon])} \le  \epsilon / (T_\epsilon e^{L_\epsilon T_\epsilon})$. 

Case 1: Assume first that $(x,t) \in B(0,1) \times [0,+\infty)$ is such that $\abs{\sigma^t(x)} \ge \epsilon$. Note that  such a $t$ is necessarily smaller than $T_\epsilon$. For all $\tau \le t$ and for all $n \ge 0$ we  have
\begin{align*}
\abs{\sigma^\tau(x)-\sigma^\tau_n(x)} &\le \int_0^\tau \abs{v(\sigma^s(x),s)-v_n(\sigma_n^s(x),s)} ds\\
&  \le \int_0^\tau \abs{v(\sigma^s(x),s)-v_n(\sigma^s(x),s)}  + \abs{v_n(\sigma^s(x),s)-v_n(\sigma_n^s(x),s)} ds\\
&  \le  \tau \norm{v-v_n}_{L^\infty(A_\epsilon \times [0,\tau])}  + L_\epsilon \int_0^\tau\abs{\sigma^s(x) -\sigma_n^s(x)} ds. 
\end{align*}
We have use the fact that  $\abs{\sigma^t(x)} \ge \epsilon$ implies that $\abs{\sigma^s(x)} \ge \epsilon$ for all $s \le \tau \le t$. We have also use the fact that, since $\abs{v} \ge \abs{v_n}$,  $\abs{\sigma_n^s(x)} \ge \abs{\sigma^s(x)} \ge \epsilon$ for all $s \le \tau \le t$ and for all $n\ge 0$.
Using  Gronwall's lemma and the fact that $t \le T_\epsilon$ we obtain that for $n\ge N$:
\begin{equation} \label{mario1}
\abs{\sigma^t(x)-\sigma^t_n(x)} \;  \le \;
T_\epsilon  \norm{v-v_n}_{L^\infty( A_\epsilon \times [0,T_\epsilon])} \;  e^{L_\epsilon T_\epsilon} \le \epsilon.
\end{equation}

Case 2: Assume that $(x,t) \in B(0,1) \times [0,+\infty)$ is such that $|x| < \epsilon$. 
Since the velocity fields $v$ and $v_n$ are focussing we clearly have that $\abs{\sigma^t(x)-\sigma^t_n(x)}  < 2 \epsilon  $  for all $n$.

Case 3: Assume finally that $(x,t) \in B(0,1) \times [0,+\infty)$ is such that $\abs{\sigma^t(x)} < \epsilon$ and $|x| \ge \epsilon$.
Since $\tau \mapsto \sigma^{\tau}(x)$ is continuous there exists a time $s \in [0, t]$ such that  
$|\sigma^s(x)|= \epsilon$. So from case 1 we get $\abs{\sigma^s(x)-\sigma^s_n(x)}\le \epsilon$ for $n \ge N$. Since  $|\sigma^s(x)|= \epsilon$ we have that $\abs{\sigma^s_n(x)}\le 2\epsilon$ for $n \ge N$. Since $s\le t$ we have $\abs{\sigma^t_n(x)} \le \abs{\sigma^s_n(x)}\le 2\epsilon$ for $n \ge N$. Therefore $\abs{\sigma^t(x)-\sigma^t_n(x)}  < 3 \epsilon  $  for all $n \ge N$.
\end{proof}

\

We are now ready to prove the Theorem \ref{thm:rd}: 

{ \it Proof of  Theorem \ref{thm:rd}.} 
Define $\rho(t):=\sigma^t \# \rho_{init}$. Recall that from Lemma~\ref{canizo} of the Appendix
$$
\calW_2(\rho,\rho_n):= \sup_{t\in [0,\infty)}  W_2(\rho(t),\rho_n(t)) \le \norm{\sigma-\sigma_n}_{L^\infty( B_1 \times [0,+\infty))}.$$ 
So from Proposition \ref{prop:unif-conv} we get that
$\calW_2(\rho_n-\rho)\to 0$. This implies in particular that for every $t \in [0,+\infty)$, $\rho_n(t)$ converges narrowly to $\rho(t)$. Since narrow convergence preserves the monotonicity (Lemma \ref{preserve} of the Appendix), we know that $\rho(t)$ is radially symmetric decreasing.

We are now going to prove  that $\rho$ and $\sigma$ satisfy \eqref{sol1aa} and \eqref{sol2aa}.
Since $\rho$ is defined by $\rho(t)=\sigma^t \# \rho_{init}$ where $\sigma^t: \real^d \to \real^d$ is the flow map associated to the velocity field $v(x,t)$, we just need to prove that 
$v(x,t)= -(\nabla K * \rho(t))(x) $ for $x \neq 0$. This is obtain by passing to the limit in the relation  $v_n(x,t)= -(\nabla K * \rho_n(t))(x)$ for $x \neq 0$.  Indeed $v_n$ converges pointwise to $v$ in $A_0 \times [0,+\infty)$. And since for fixed $t$, $\rho_n(t)$ converges narrowly to $\rho(t)$, we obtain from (\ref{formula}) that $\nabla K * \rho_n$  converges pointwise to $\nabla K * \rho$  in $A_0 \times [0,+\infty)$.
\qed

\section{Instantaneous mass concentration}
\label{instant}

\

 The recent work of \cite{BLR,HJD} concerns local well-posedness of the problem with initial data in $L^p$. One can prove
a sharp condition on $p$ for local well-posedness by considering a family of initial data that behave as a powerlaw near the origin. Such initial conditions satisfy the monotonicity assumptions considered in this paper.
In this section, using existence results from the prior section
and a bootstrap argument, we
prove results about the behavior of these solutions 
as measure solutions that concentrate mass.  Such results are not discussed in the prior literature for the singular power law potential $K(x)=|x|^\alpha$, $\alpha<1$. 
%

More specifically in \cite{BLR} it was proven that the aggregation equation with potential  $\nabla K(x)= x\abs{x}^{\alpha-2}$, $2-d<  \alpha<2$, is locally well posed in any $L^p$-space with $p>\frac{d}{d + \alpha -2}$.   Note that  given $\beta \in (\frac{d+\alpha-2}{d},1)$ the function  
 $$h(x)=
 \begin{cases}
 \frac{c}{|x|^{{d + \alpha-2}}}\;  \frac{1}{(-\log\abs{x})^{\beta}}& \text{ if } \abs{x} \le 1\\
 0 & \text{ otherwise}
 \end{cases}
 $$
  belongs to the critical space $L^{\frac{d} {d + \alpha -2}}(\real^d)$ but does not belong to any $L^p$ space with $p>\frac{d}{d + \alpha -2}$.
 In \cite{HJD} it was proved that if the initial data is exactly equal to $h(x)$ then a solution of the aggregation equation instantaneously leaves the space $L^{\frac{d} {d + \alpha -2}}$. 
  In this section we go a
little further and show that the solution not only leaves $L^{{{d} \over {d
+ \alpha -2}}}$ but also instantaneously concentrates some point mass at the
origin.  Our results make use of the existence theory from the previous section.
 Also compared to the  work in \cite{BLR} and \cite{HJD},  our argument here is local in essence and holds  for any radially symmetric decreasing initial data which is locally more singular than $h(x)$ at the origin. The main theorem of the section is the following:
\begin{theorem} \label{main-theorem} Let $\nabla K(x)=x \abs{x}^{\alpha-2}$, $2-d<  \alpha<2$.
Suppose $\rho_{init} \in \prob_{RD}(\real^d)$  is compactly supported and absolutely  continuous with respect to the Lebesgue measure.
 Suppose that there exists $c>0$, $r_0>0$ and $\beta \in (\frac{d+\alpha-2}{d},1)$ such that the
 density $u_{init}$ of $\rho_{init}$ satisfies
\begin{equation} 
u_{init}(x) \ge \frac{c}{|x|^{d + \alpha-2}}\;  \frac{1}{(-\log\abs{x})^{\beta}} \quad \text{ for all } |x|<r_0.
\end{equation}
Suppose finally that $\rho \in C([0,+\infty), \prob_{RD}(\real^d))$ satisfies the Lagrangian formulation \eqref{sol1aa}-\eqref{sol2aa} of the aggregation equation. Then  $\rho(t)(\{0\})>0$ for all $t>0$. 
\end{theorem}



%

%

\subsection{Comparison principles}

\

In this subsection we derive a few comparison principles which will be necessary in order to make the arguments local.

\begin{lemma} \label{vel-comp}
Suppose $v_1,v_2 \in \vel$ and  $\abs{v_1} \ge \abs{v_2}$. Then 
$$
\sigma_1^t  \# \mu \succ \sigma_2^t \# \mu \qquad \text{ for all  } \mu \in \prob_R(\real^d) \text{ and } t \ge 0
$$
where  $\sigma_1$ and $\sigma_2$ are the  flow maps associated to $v_1$ and $v_2$ respectively.  

\end{lemma}
\begin{proof} Since $v_2 \in  \vel$ the flow map $\sigma_2^t$ is invertible away from the origin.   Define $\Lambda^t_2(x)=(\sigma_2^t)^{-1}(x)$ if $x \neq 0$ and $\Lambda_2^t(0)=0$. One can then easily check that
$$(\sigma_1^t \circ \Lambda_2^t) \# (\sigma_2^t\# \mu)= \sigma_1^t\# \mu$$
Moreover since $\abs{v_1} \ge \abs{v_2}$ we have that   $\abs{(\sigma_1^t \circ \Lambda_2^t)(x)} \le \abs{x}$, which concludes the proof.
\end{proof}


\begin{lemma}
Suppose $v \in \vel$. Suppose also that  $\mu, \nu \in \prob_R(\real^d)$ and  $\mu \succ \nu$. Then 
$$
\sigma^t  \# \mu \succ \sigma^t \# \nu \qquad \text{ for all } t \ge 0
$$
where  $\sigma$ is the  flow maps associated to $v$. 
\end{lemma}
\begin{proof}
Since $\mu \succ \nu$ there is a map $P$  satisfying   $\abs{P(x)} \le \abs{x}$ such that $\mu=P \#\nu$. As in the previous lemma,  define $\Lambda^t(x)=(\sigma^t)^{-1}(x)$ if $x \neq 0$ and $\Lambda^t(0)=0$.  One can then easily check that
$$(\sigma^t \circ P \circ \Lambda^t) \#  (\sigma^t \# \mu)=\sigma^t \# \nu$$
and $\abs{(\sigma^t \circ P \circ \Lambda^t)(x)} \le \abs{x}$ which conclude the proof.
\end{proof}



The following definition will be needed in order to compare two measures of different mass.
\begin{definition} Suppose $\rho \in \prob_R(\real^d)$ and $\mu \in \mes_R(\real^d)$,  with $\mu(\real^d) \le 1$. We write $\rho \triangleright \mu$ if  there exists a measure  $\nu \in \prob_R(\real^d)$ 
such that
$$
\rho \succ \nu \quad \text{and} \quad \nu(A) \ge \mu(A) \;\; \forall A \in \calB(\real^d).
$$
\end{definition}
In view of \eqref{comp1} and \eqref{comp2} it is clear that:
\begin{equation}  \label{comp3}
\rho \triangleright \mu  \quad  \Longrightarrow  \quad   \abs{\nabla K * \rho} \ge \abs{\nabla K * \mu } 
\end{equation}
The following Lemma will be useful in order to make localized comparisons.  
\begin{lemma} \label{comp-lemma}
Suppose $v_1, v_2 \in \vel$ and $\abs{v_1} \ge \abs{v_2}$ in  $B(0,2R) \times [0,+\infty)$ .
  Suppose also that $\rho \in \prob_R(\real^d)$, $\mu \in \mes_R(\real^d)$ and  $\rho \triangleright \mu$.  Then
   $$\sigma^t_1 \# \rho \;\;  \triangleright \; \; \sigma^t_2  \# (\mu \, \chi_{B(0,R)}) \qquad \text{ for all } t \ge 0,  $$   
   where  $\sigma_1$ and $\sigma_2$ are the  flow maps associated to $v_1$ and $v_2$ respectively, and $\chi_{B(0,R)}$ is the indicator function of the set $B(0,R)$. 
\end{lemma}

\begin{proof}
Since $\rho \triangleright \mu$ there exists a probability measure $\nu$ such that 
$\rho \succ \nu \ge \mu$.
Let $\xi(x)$ be a smooth radially symmetric function which satisfies $\xi(x)=1$ if $\abs{x} \le R$, $ \xi(x)=0$ if $\abs{x} \ge 2R$ and $\chi(x) \le 1$ for all $x \in \real^d$. The velocity field $v_3(x,t):=v_2(x,t) \xi(x)$ is still in $\vel$. Moreover we have $\abs{v_3} \le \abs{v_1}$ for all $x \in \real^d$ and $t \ge 0$. We can therefore use the two previous Lemmas to obtain that
$$
\sigma_1^t \# \rho \succ \sigma_3^t\# \rho \succ  \sigma_3^t \#\nu \ge \sigma_3^t\# \mu \ge
\sigma^t_3  \# (\mu \, \chi_{B[0,R]})
$$
The last two inequalities are a simple consequence of the definition of the push-forward together with the fact that $\nu \ge \mu \ge \mu \, \chi_{B[0,R]}$. Finally, note that since $v_3=v_2$ on $B(0,R) \times [0,+\infty)$, then $\sigma^t_3  \# (\mu \, \chi_{B[0,R]})=\sigma^t_2  \# (\mu \, \chi_{B[0,R]})$.
\end{proof}

\subsection{Proof of Theorem \ref{main-theorem} by bootstrap argument}

\

Fix $\alpha \in (2-d, 2)$ and define the functions
\begin{align}
&f_{\epsilon,r_0}(x) = \frac{1}{\abs{x}^{d+\alpha -2 +\epsilon}}  \;\;  \chi_{B(0,r_0)}(x) \qquad  \text{for }\epsilon \in (0,1)\\
&g_{r_0}(x) = \frac{1}{\abs{x}^{d + \alpha -2}}  \;\;   \chi_{B(0,r_0)}(x)  \\
&h_{\beta, r_0}(x)= \frac{1}{\abs{x}^{d + \alpha -2}}
 \frac{1}{(- \ln\abs{x})^{\beta}}  \;\; 
  \chi_{B(0,r_0)}(x)  \qquad \text{for }\beta \in (\frac{d+\alpha -2}{d},1).
\end{align}
Note that  at the origin $f_{\epsilon,r_0}$ is more singular than $g_{r_0}$ which itself is more singular than
  $h_{\beta, r_0}$.
  In \cite{BLR} it was proved
 that  if $\alpha =1$ and the initial data is exactly
 equal to $Cf_{\epsilon,r_0}(x)$ ($C$ is a normalizing constant) then a
Dirac delta function  appears instantaneously in the
 solution. The proof relied on the fact that solutions of
 the ODE $\dot{x}=-(\nabla K * f_{\epsilon,r_0})(x)$ reach
 the origin in finite time. However this strategy does not
 work with $g_{r_0}$ and $f_{\epsilon,r_0}$,
because  solutions of  $\dot{x}=-(\nabla K * g_{r_0})(x)$
 and $\dot{x}=-(\nabla K * h_{\beta,r_0})(x)$ do not reach
 the origin in finite time.
For that reason  we will use a bootstrap argument
 to prove that a delta function appears instantaneously
 when the  initial data is equal to
or  more singular than
 $h_{\beta,r_0} \in L^{\frac{d}{d + \alpha -2}}(\real^d)$.
 Roughly speaking, we
 will show that the velocity field $-\nabla K * h_{\beta,r_0}$
 instantaneously deforms $h_{\beta,r_0}$ into a
 function more singular than $g_{r_0}$, then we will
 show that the velocity field $-\nabla K * g_{r_0}$
 instantaneously  deforms $g_{r_0}$ into a function
 more singular than $f_{\epsilon,r_0}$, and finally
 we will use  the argument from \cite{BLR} to show that the
 velocity field $-\nabla K * f_{\epsilon,r_0}$
 deforms $f_{\epsilon,r_0}$ in such a way that a
 delta function  appears instantly.

The following definition is consistent
 with Definition \ref{definition:radial}:
\begin{definition}
Given a radially symmetric, non-negative
 function $u\in L^1(\real^d)$, we define $\hat{u} \in L^1((0,+\infty))$ to
 be the unique function satisfying
 $$ \int_{r_1}^{r_2} \hat{u}(r) dr=\int_{r_1<\abs{x}<r_2} u(x) dx \qquad \text{for all } r_1,r_2 \ge 0. $$
 In other words, $
\hat{u}(r)=u(r) \omega_d r^{d-1}.
$
 \end{definition}

With this notation we have:
 \begin{align} \label{vr}
&\hat{f}_{\epsilon,r_0}(r)=  \omega_d\;\;   \frac{1}{r^{\alpha -1 + \epsilon}}\;\;  \chi_{[0,r_0]}(r) \\ \label{vrr}
&\hat{g}_{r_0}(r) = \omega_d \;\; \frac{1}{r^{\alpha -1}} \chi_{[0,r_0]}(r) \\
&\hat{h}_{\beta,r_0}(r)= \omega_d \;\;
 \frac{1}{r^{\alpha -1}} \frac{1}{(- \ln r)^{\beta}}  \;\;  \chi_{[0,r_0]}(r)\label{vrrr}.
\end{align}
 We remind the reader that by Lagrangian solution we mean a function $\rho(x,t)$ that satisfies the Lagrangian formulation \eqref{sol1aa}-\eqref{sol2aa} of the aggregation equation. 

\begin{proposition} \label{vel-est} Let $\rho \in C([0,+\infty), \prob_{RD}(\real^d))$ be a  Lagrangian
  solution of the aggregation
 equation with compactly supported initial data $\rho_{init}$ and
 potential $K$ satisfying  $\nabla K(x)=x \abs{x}^{\alpha-2}$, $2-d < \alpha <2$. Let $v(x,t)= -(\nabla K * \rho(t))(x)$ be the associated velocity field.
\begin{enumerate}
 \item[(i)] If $\rho_{init} \triangleright  c f_{\epsilon,r_0} $ for some $c,r_0>0$ and $\epsilon \in (0,1)$,  then there exist $R,C>0$ such that
  $$
\abs{v(x,t)} \ge C \abs{x}^{1-\epsilon} \qquad \text{for all } (x,t)\in B(0,R)\times [0+\infty).
$$
\item[(ii)] If $\rho_{init} \triangleright  c\; g_{r_0} $ for some $c,r_0>0$,  then there exist $R,C>0$ such that
$$
\abs{v(x,t)} \ge C \; \abs{x} (-\ln\abs{x} ) \qquad \text{for all } (x,t)\in B(0,R)\times [0+\infty).
$$
\item[(iii)] If $\rho_{init} \triangleright  c h_{\beta,r_0}$  for
 some $c,r_0>0$ and $\beta \in (\frac{d + \alpha -2}{d},1)$, then there exist $R,C>0$ such that
$$
\abs{v(x,t)} \ge C \; \abs{x} (-\ln\abs{x} )^{1-\beta} \qquad \text{for all } (x,t)\in B(0,R)\times [0+\infty).
$$
\end{enumerate}
\end{proposition}
\begin{proof}
Let us prove (i). On one hand, from \eqref{comp3}  we see
 that $\abs{v(x,0)} \ge c \abs{(\nabla K*  f_{\epsilon,r_0})(x)}$
 for all $x\in \real^d$. On the other hand, since the velocity field is always pointing inward (this is due to the positivity of $\phi$), we have that  $\rho(t) \succ \rho(0)$ for all $t \ge 0$, and therefore from
 \eqref{comp2}  we get 
 that $\abs{v(x,t)} \ge \abs{v(x,0)}$
for all $x\in \real^d$ and $t\ge 0$. So
 we only  need to show that
 $\abs{(\nabla K*  f_{\epsilon,r_0})(x)} \ge C \abs{x}^{1-\epsilon}$
 in some neighborhood of the origin, and this
  estimate follows easily from 
Lemma \ref{dong}.  Indeed, by Lemma \ref{dong} we have for $|x|\le r_0$
\begin{align}
\abs{\nabla K * f_{\epsilon, r_0} (x)} &=  \omega_d 
\int_0^{|x|} \phi\p{\frac{r}{\abs{x}}}  \p{\frac{\abs{x}}{r}}^{\alpha-1} r^{- \epsilon} dr \\ & \qquad \qquad \qquad+  \omega_d  |x| \int_{|x|}^{r_0} 
 \phi\p{\frac{r}{\abs{x}}}  \p{\frac{r}{\abs{x}}}^{2-\alpha} r^{-1-\epsilon} dr \\
&\geq\omega_d  C_1  
\int_0^{|x|}r^{- \epsilon} dr + \omega_d  C_2  |x| \int_{|x|}^{r_0} 
r^{-1-\epsilon} dr \label{vivi} \\
& \ge \omega_d  C_1 \frac{|x|^{1-\epsilon}}{1-\epsilon}
\end{align}
where $C_1= \inf _{[0,1]}\phi=\phi(1)$ and $C_2= \inf_{(1,+\infty)} \phi(r)r^{2-\alpha}>0$.

 Let us now prove (ii). Reasoning as
 above we see that  it is enough  to
 show that $\abs{(\nabla K*  g_{r_0})(x)} \ge C \; \abs{x} (-\ln\abs{x} ) $
 in some neighborhood of the origin.
 Then the argument is similar.  From \eqref{vivi} with $\epsilon=0$ we get
 $$\abs{\nabla K * g_{r_0} (x)} \ge 
 \omega_d  C_2 \abs{x} \ln\p{\frac{r_0}{\abs{x}}}
 $$
 which yields to the desired estimate.

 To prove (iii) it is enough to
 show $\abs{(\nabla K*  h_{\beta ,r_0})(x)}
 \ge C \; \abs{x} (-\ln\abs{x} )^{1-\beta}$  in some neighborhood
 of the origin, and the argument is similar.  In this case we have
$$|\nabla K  * u_0(x)| \ge  \omega_d  C_2|x| \int_{|x|}^{r_0} \frac{1}{|\log r|^{\beta}}
{\frac{dr}{r}},$$
\noindent which yields to the desired estimate.
 This last  estimate was derived independently  in \cite{HJD}.
\end{proof}
           
The ODE's $$\dot{r}=-C r^{1-\epsilon}, \quad \dot{r}=-C r (-\ln r ) \quad
 \text{ and } \quad \dot{r}=-C r (-\ln r )^{1-\beta}$$ suggested by the previous  proposition have explicit solutions and their flow maps  are respectively:
\begin{align}
\label{a}
&\sigma_1^t(r) = {\sigma_1}(r,t)=
\begin{cases}
(r^\epsilon-\epsilon Ct)^{1/\epsilon}& \text{ if } r>(\epsilon Ct)^{1/\epsilon}\\
0 &\text{ if } r\le (\epsilon Ct)^{1/\epsilon}
\end{cases} \\
\label{b}
&\sigma_2^t(r) = {\sigma_2}(r,t)=r^{e^{Ct}} 
\\
\label{c}
&\sigma_3^t(r) = {\sigma_3}(r,t)=e^{-\p{C\beta t +
 \p{\ln{\frac{1}{r}}}^\beta}^{1/\beta}}  
\end{align}
Solutions of the first ODE reach the origin in finite time but solutions of the other two ODE's only approach  the origin as $t \to \infty$.
 Corresponding to the flow maps  $\sigma_i: [0,+\infty) \times [0,+\infty) \to [0,+\infty)$ there are  flow maps $S_i: \real^d \times [0,+\infty) \to \real^d$ defined by $S_i(x,t)=\sigma_i(\abs{x},t)\frac{x}{\abs{x}}$. 
 The $S_i$ are the flow maps associated to the
 velocity fields $w_1(x)=-C \abs{x}^{1-\epsilon} \frac{x}{\abs{x}}$, $w_2(x)=-C \abs{x} (-\ln\abs{x} ) \frac{x}{\abs{x}}$,
 and $w_3(x)=-C \abs{x} (-\ln\abs{x} )^{1-\beta}\frac{x}{\abs{x}}$. Let  $u\in L^1(\real^d)$ be a radially symmetric, non-negative function. It is clear from \eqref{a} that $S^t_1 \# u$ has a point mass at the origin if $u$ has non-zero mass in $B(0,(\epsilon Ct)^{1/\epsilon})$. On the other hand, because $S^t_2$ and  $S^t_3$ are smooth invertible maps, $S^t_2 \# u$ and $S^t_3 \# u$
are continuous with respect to the Lebesgue measure, and by the change of variable formula, we have
\begin{gather} \label{push-formula}
(S_i^t \# u)\sphat \;(r) =  
(\sigma_i^t \# \hat{u})(r)  =  \hat{u}(\tau_i^t(r)) \;\;  \der{\tau_i^t (r)}{r}  \qquad i=2,3 \\
\text{ where }\tau_i^t(r)= (\sigma_i^t)^{-1}(r)
\end{gather}

\begin{proposition}[Bootstrap] 
Let $\rho \in C([0,+\infty), \prob_{RD}(\real^d))$ be a Lagrangian
  solution of the aggregation
 equation with compactly supported initial data $\rho_{init}$ and
 potential $K$ satisfying $\nabla K(x)=x \abs{x}^{\alpha-2}$, $2-d < \alpha <2$. 
\begin{enumerate}
\item[(i)] If $\rho_{init} \triangleright   c f_{\epsilon, r_0} $  for some $c,r_0>0$ and $\epsilon \in (0,1)$,  then $\rho(t)(\{0\})>0$ for all $t >0$.
\item[(ii)] If $\rho_{init} \triangleright  c\; g_{r_0} $  for some $c,r_0>0$, then for any $t>0$ there exist constants $c_1,r_1>0$ and $\epsilon \in (0,1)$, such that $\rho(t) \triangleright  c_1\, f_{\epsilon,r_1} $. 
\item[(iii)] If $\rho_{init} \triangleright c \;h_{\beta,r_0}$, for
 some $c,r_0>0$ and $\beta \in (\frac{d +\alpha -2}{d},1)$, then
 for any $t>0$ there exists 
 constants $c_1,r_1>0$ such that $\rho(t) \triangleright   c_1\; g_{r_1} $. 
\end{enumerate}
\end{proposition}
\begin{proof}  Let us prove (i).  Let $S_1^t(x)=S_1(x,t)$ be the flow map generated by the velocity field $w_1(x)=-C \abs{x}^{1-\epsilon} \frac{x}{\abs{x}}$ suggested by  Proposition  \ref{vel-est}.  From Lemma \ref{comp-lemma} and Proposition  \ref{vel-est} we then obtain  that
\begin{equation} \label{ugh}
\rho(t) \triangleright S_1^t \# c f_{\epsilon, r_1} \end{equation}
 for $r_1$ small enough and for all $t \ge 0$. 
 Let us fix a $t>0$. Since $f_{\epsilon, r_1}$ has non-zero mass in  $B(0,(\epsilon Ct)^{1/\epsilon})$, it is clear  from  \eqref{a} that 
  the measure $S_1^t \# c f_{\epsilon, r_1}$ has a point mass at the origin. Then by \eqref{ugh}  we conclude that $\rho(t)$ also has a point mass at the origin. 

Let us now prove (ii). 
 Again Lemma \ref{comp-lemma} and Proposition \ref{vel-est} imply that 
$\rho(t) \triangleright S_2^t \# c g_{ r_1}$  for $r_1$ small enough. As already mentioned $S_2^t \# c g_{ r_1}$ is continuous with respect to Lebesgue measure. We are going to show that given any $t>0$,  $S_2^t \#  g_{ r_1} \ge  c_2 f_{r_2,\epsilon}$ for some constant $c_2,r_2>0$  and $\epsilon \in (0,1)$ which will conclude the proof of (ii).
Let 
$
\tau_2^t(r)= (\sigma_2^t)^{-1}(r)=r^{e^{-ct}}
$
where $\sigma_2^t(r)$ is defined by \eqref{b}. Using the change of variable formula, we get that
\begin{align*}
(\sigma_2 \# \hat{g}_{r_1})(r)&=\hat{g}_{r_1}(\tau_2^t(r)) \;\; \der{\tau_2^t}{r}(r)\\
&= \frac{\omega_d}{\p{r^{e^{-ct}}}^{\alpha-1}} e^{-ct} r^{e^{-ct}-1} \quad \text{for $r$ small enough}\\
&= \frac{\omega_d}{r^{\alpha-1 + (2-\alpha)(1-e^{-ct})}}  e^{-ct}
\end{align*}
Since $2-\alpha>0$ it is clear that $(\sigma_2 \# \hat{g}_{r_1})(r) \ge  c_2 \hat{f}_{r_2,\epsilon}(r)$ for $r$ small enough.

Let us now prove (iii).  Once more Lemma \ref{comp-lemma} and Proposition \ref{vel-est} imply that 
$\rho(t) \;  \triangleright \;  S_3^t \# c h_{\beta, r_1}$  for $r_1$ small
 enough. Let us fix $t>0$ and  show
 that $S_3^t \# c h_{\beta, r_1} \ge c_2 g_{r_2}$ for
 $r_2$ small enough. In view of \eqref{vrr} it is enough
 to prove that 
 \begin{equation} \label{vruum}
 \lim_{r \to 0} r^{\alpha-1} \p{\sigma_3^t \#  \hat{h}_{\beta, r_1}(r)}>0.
\end{equation}
Let $\tau_3^t(r)=(\sigma_3^t)^{-1}(r)$ and note that
\begin{equation} \label{vroum}
\ln \frac{1}{\tau_3^t(r)}= \p{-c \beta t+\p{\ln \frac{1}{r}}^\beta }^{1/\beta}
\end{equation}
From now on we drop the lower subscript.  From the change of variable formula we have
\begin{align}
\sigma^t \# \hat{h}(r)&= \hat{h}(\tau^t(r)) \; \der{\tau^t}{r}(r)\\
&= \frac{\tau^t(r)^{1-\alpha}}{\p{\ln \frac{1}{\tau^t(r)}}^\beta}  \der{\tau^t}{r}(r) \label{ff1}\\
&= \frac{\tau^t(r)^{2-\alpha}}{-c \beta t+\p{\ln \frac{1}{r}}^\beta} \frac{ \der{\tau^t}{r}(r)}{\tau^t(r)} \label{ff2}
\end{align}
where we have used \eqref{vroum} to go from \eqref{ff1} to \eqref{ff2}. Then note that using \eqref{vroum} again we get
$$
 \frac{ \der{\tau^t}{r}(r)}{\tau^t(r)}= 
 - \der{}{r} \ln\p{\frac{1}{\tau^t(r)}}=
 \p{-c \beta t + \p{\ln \frac{1}{r}}^\beta}^{\frac{1}{\beta}-1} \p{\ln \frac{1}{r}}^{\beta-1} \frac{1}{r}
$$
which combined with \eqref{ff2} gives
\begin{align}
r^{\alpha-1} \p{\sigma^t\# \hat{h}(r)}&= \p{\frac{\tau^t(r)}{r}}^{2-\alpha}  \p{-c \beta t + \p{\ln \frac{1}{r}}^\beta}^{\frac{1}{\beta}-2} \p{\ln \frac{1}{r}}^{\beta-1}\\
&= \p{\frac{\tau^t(r)}{r}}^{2-\alpha}   \p{1- \frac{c \beta t}{ \p{\ln \frac{1}{r}}^\beta}}^{\frac{1}{\beta}-2} \p{\ln \frac{1}{r}}^{-\beta}\\
& \ge \frac{1}{2}\p{\frac{\tau^t(r)}{r}}^{2-\alpha}   \p{\ln \frac{1}{r}}^{-\beta}  \quad \text{for $r$ small enough} \label{vrim}
\end{align}
Using \eqref{vroum} and doing a Taylor expansion we find that
\begin{align}
\ln \p{\frac{\tau^t(r)}{r}}&= 
\ln \p{\frac{1}{r}} - \ln \p{\frac{1}{r}} \p{1- \frac{c \beta t}{ \p{\ln \frac{1}{r}}^\beta}}\\
&= ct \p{\ln \frac{1}{r}}^{1-\beta} \p{ 1 + o\p{\frac{c \beta t}{ \p{\ln \frac{1}{r}}^\beta}}} \\
& \ge \frac{1}{2}ct \p{\ln \frac{1}{r}}^{1-\beta}  \quad \text{for $r$ small enough} \label{vram}
\end{align}
Combining \eqref{vrim} and \eqref{vram} we get 
$$
\ln\p{r^{\alpha-1} \p{\sigma^t \# \hat{h}(r)}} \ge 
\ln (1/2) + \frac{1}{2}(2-\alpha)ct \p{\ln \frac{1}{r}}^{1-\beta} - \beta \ln \ln \frac{1}{r} 
$$
for $r$ small enough. Since $2-\alpha>0$ it is clear that $\lim_{r \to \infty}\ln\p{r^{\alpha-1} \p{\sigma^t \# \hat{h}(r)}}=+\infty$ which implies \eqref{vruum}. 
\end{proof}

We now prove Theorem \ref{main-theorem}:

{ \it Proof of Theorem \ref{main-theorem}.}
If $\rho_{init} \triangleright c \;h_{\beta,r_0}$ for some
 for some $c,r_0>0$ and $\beta  \in (\frac{d + \alpha -2}{d},1)$, we can apply the previous proposition to get that   for any $t_1>0$, $\rho(t_1) \triangleright   c\; g_{r_0}$ for some different constants $c,r_0>0$. Applying the proposition again, we get  that   for any $t_2>t_1$, $\rho(t_2) \triangleright   c\; f_{\epsilon,r_0}$ for some other constants $c,r_0>0$ and for some $\epsilon \in (0,1)$. Applying the proposition one last time we get that   for any $t_3>t_2$,  $\rho(t_3)(\{0\})>0$. Since $t_1<t_2< t_3$ can be chosen arbitrarily small, this conclude the proof.
\qed


\section{Newtonian potential case}
\label{Newtonian}

\

An even more singular case is that of the Newtonian potential, $\alpha = 2-d$ in general, with $K(x) = \log|x|$ in the special case of 2D.
Without loss of generality we use the normalization for $K$ that yields $\Delta (K*\rho)= \rho$, i.e. the fundamental solution of the Poisson equation.  This simple fact localizes the dynamics as compared to the nonlocal case studied in previous sections.  In Eulerian coordinates, for smooth densities, we have 
\begin{equation}\rho_t + v\cdot \nabla \rho = \rho^2. \label{rhoeq}
\end{equation}

Recall  that for radially symmetric problems, the Laplace operator is
$$\Delta f = \frac{1}{r^{d-1}} \frac{\partial}{\partial r}\bigl( r^{d-1} \frac{\partial f}{\partial r}\bigr).$$
Likewise we have the following formulae for the gradient and divergence operators:

$$\nabla f = \frac{\partial f}{\partial r} \vec r,$$ where $\vec r$ is the unit
outward pointing radial vector and 
$$div\  v =  \frac{1}{r^{d-1}} \frac{\partial}{\partial r} r^{d-1} v.$$

Using the latter formula we can rewrite $v$ above in terms of $\rho$ simply by inverting
$div v = -\rho$:
\begin{equation}v(r) = -\frac{1}{r^{d-1}} \int_0^r s^{d-1} \rho(s) ds : = - \frac{m(r)}{r^{d-1}},\label{radialomega}
\end{equation}
where $m(r)$ is proportional to the mass contained inside a ball of radius $r$.
Thus it makes sense to rewrite the evolution equation (\ref{rhoeq}) in mass coordinates - 
in general regardless of the kernel it is
\begin{equation}
m_t + vm_r = 0.
\label{mass}
\end{equation}
 However this greatly simplifies in the special Newtonian case.  Formula (\ref{radialomega}) gives
$$m_t - \frac{m m_r}{r^{d-1}} = 0. $$   By changing variables to $z$ coordinates, where $z=\frac{r^d}{d}$, we have the inviscid Burgers equation,
\begin{equation}
m_t -m m_z=0.\label{burger}
\end{equation}

The transformation to equation (\ref{mass}) is well-known; the transformation to the $z$ variable appeared in \cite{Biler} in the context
of a viscous version of our problem arising in astrophysics.  
Here we use the classical conservation law theory for the inviscid (purely transport) problem to prove that monotonicity is preserved
by the flow in all dimensions.
In one dimension it is known that $K=|x|$ can be transformed to the inviscid Burgers problem see e.g. \cite{BV}.
The connection to Burgers equation allows us to prove quite a lot about radially symmetric solutions of the aggregation equation with Newtonian potential, by directly connecting to the classical theory of conservation laws.  We consider three cases: (a) monotone decreasing radial densities for which we have a unique forward time solution; (b) general radial densities for which we have existence of solutions but uniqueness requires
the specification of a jump condition (akin to choosing a particular entropy-flux pair for the definition of distribution solution, and; (c) the case of radially symmetric signed measures for which one requires an additional entropy condition in order to have a unique solution.  All of these cases can be distinguished by the known properties of the inviscid Burgers equation  \cite{Bardos,Lax,Evans}.

\subsection{Case 1: $\rho\in\prob_{RD} (\R^d)$  - existence of unique classical solutions}

\

In the case of radially symmetric monotone decreasing probability measures, we have unique classical solutions by virtue of the fact that the corresponding flow field $v$ is Lipschitz for $r>0$.
Monotonicity is preserved by virtue of the localization of the equations as described above.
More specifically, we have the heuristic that $\rho$ satisfies $\rho_t = \rho^2$ along characteristics so the initial ordering of the density is preserved provided that the characteristics remain well ordered and are well defined.  We can prove this to be the case by going to the mass coordinate formulation above.  The condition that the characteristics remain well defined is akin to proving that shocks will not form from any initial data satisfying the monotonicity condition.
If a shock forms - which we define as a singularity in $m_z$ in the mass equation (\ref{burger}), 
the first time of formation will occur at $t_{shock} = 1/ \sup_z \{m_{init}'(z)\}$.  So we need the characteristic to reach the origin before this time occurs.  Denote by $z_s$ the location at time zero
of this characteristic.  Then our condition on the shock occurring after the characteristic crosses zero is 
$$\frac{z_s}{m_{init}(z_s)} < \frac{1}{m_{init}'(z_s)}\iff m_{init}(z_s)>z_s m_{init}'(z_s)$$ since both $m_{init}$ and $m_{init}'$ are nonnegative.  Using the definition of the mass $m$ and converting back to regular radial coordinates, the above is equivalent to the following condition on the density $\rho$:
\begin{equation}  \int_{B(0,R)}  \rho_{init}(R) dx \leq \int_{B(0,R)} \rho_{init} (x) dx 
\label{equalssign}
\end{equation}
for all $R$, which is true for monotone decreasing initial data $\rho_{init}$.  
The special case of the equals sign in (\ref{equalssign}) corresponds to the shock happening right when the characteristic reaches the origin.
There are exact solutions that satisfy this - corresponding to a density that is the characteristic function of a collapsing
ball.  
The corresponding solution in $(m,z)$ coordinates is 
the well-known Burgers solution
of the form $-z/(1-t)$ that forms a shock in finite time in which all of the characteristics on an interval collapse at the origin simultaneously. This example is the most singular case of the general class of solutions considered
in this subsection. Since the only shocks that form occur at the origin, which is a boundary of the domain,
this results in a global-in-time classical solution of (\ref{burger}) for any initial condition $m_{init}(z)$ arising from a probability density $\rho_{init} \in \prob_{RD}(\R^d)$.  The classical solution 
of the inviscid Burgers equation easily gives us a unique solution of the Lagrangian formulation of the problem as well.  We state these results below:
\begin{theorem}
Given compactly supported initial data $\rho_{init}\in\prob_{RD} (\R^d)$, define $m_{init} = \int_0^r s^{d-1}  d\rho$.  
Then there exists a unique classical solution to equation (\ref{burger}) on the half space $(x,t) \in (0,\infty)\times [0,\infty)$ and a corresponding unique solution of the Lagrangian mapping formulation of the density transport problem.  The solution retains its monotonicity property for all time.
\end{theorem}

 As we show below the situation is much more complicated for general radially solutions without the monotonicity condition.
Some observations can be immediately made using classical results from conservation laws.  Moreover these results connect directly to related
problems in fluid dynamics such as vortex sheet solutions of the 2D Euler equations.  The next two subsections provide a discussion of
these observations.

\subsection{Case 2: $\rho\in\prob_{R} (\R^d)$  - existence of unique solutions with jump condition}

\

In the case of general radially symmetric probability densities, we no longer have classical solutions.  Let us consider the simplest example of data that violates the monotonicity condition - that of a uniform delta-concentration on the boundary of the ball of radius $R_*$.
Following the mass coordinates, we see that this example has a jump discontinuity in $m(z)$ at
$z^*= (R^*)^d/d$.  Since $m$ is the characteristic speed, this results in a jump in the velocity across the delta-ring.  One way to define the solution is to consider a distribution
solution of (\ref{burger}) in which case the speed of the shock (velocity of the delta-ring) is
defined, in $z$ coordinates as $s_{z-shock} = (m_1+m_2)/2$, i.e. the Rankine-Hugoniot condition associated with equation (\ref{burger}).  
As is well-known for scalar conservation laws, we could transform equation (\ref{burger}) by multiplying by any function of $m$, 
\begin{equation}
(F(m))_t -(G(m))_z=0, \quad F'(m) = f(m), \quad G'(m) = m f(m)\label{burgerf}
\end{equation}
for some function $f$, 
yieldling a different jump condition in the weak-distribution form of (\ref{burgerf}), 
$$s_{z-shock} = \frac{[F(m)]}{[G(m)]}, $$where $[\phantom{F}]$ denotes the jump across the shock.

By virtue of well-known results for scalar conservation laws, we obtain families of weak solutions for the general radially symmetric problem.  For a given formulation of the form (\ref{burgerf}) there exists a unique distribution solution.  Uniqueness for inviscid Burgers often requires an additional entropy condition.  In the case of formulation
(\ref{burger}), the entropy condition is automatically satisfied by the monotonicity of $m$, which is guaranteed for any radial probability density, not necessarily monotone.  The full entropy condition would only be required in the case of non-monotone $m$ such as would arise in the case of a signed measure $\rho$. 

 It would be interesting to know whether there is an optimal choice of shock speeds for these under-determined problems.  For example one might also consider an optimal transport framework in which the best choice of shock speed would be one in which
the interaction energy is most quickly dissipated.  For the aggregation problem this would result
in the fastest speed possible for the delta-ring which would satisfy a Lagrangian formulation
of the problem but perhaps not a classical distribution solution in Eulerian variables - even in the $m-z$ framework described above. 
We note that the entropy solution discussed above, in which the speed of the shock is chosen to be the average of the speeds on either side, is a natural generalization of the choice conventionally made for 2D vortex sheets, in which
$v=\nabla^{\perp} K*\rho$ rather than $v=\nabla K*\rho$, and $\rho$ is the vorticity. For that problem
one ascribes a velocity to the sheet that is the arithmetic average of the
speeds on either side \cite{MB}. The frozen time calculation can be made by analogy
to the incompressible flow problem, however the ensuing dynamics is quite different.  For the vortex sheet problem
the flow is tangential to the sheet so the issue of shocks does not arise.  For the aggregation problem the flow
is normal to the sheet and affects the solution on either side of it, because the speed of the shock determines
the rate at which characteristics on either side of the discontinuity are absorbed and the rate at which information is lost in the discontinuity.
 To summarize, if we define a solution as satisfying
an equation of the form (\ref{burgerf}) in the sense of distributions, then we expect a unique solution, however, in the case of jump discontinuities in $m$, the shock speed will depend on the choice of entropy-flux pair as discussed above.  Moreover we believe even more general examples may exist that could satisfy an optimality condition associated with dissipation of the interaction energy.  We finally briefly mention the case of signed measures below.

\subsection{Case 3: signed measures}

\

The case of signed measures introduces yet another source of nonuniqueness of solutions, which we briefly discuss.
A signed measure corresponds to a non-montone (but $L^\infty$) solution of the inviscid Burgers problem.  This general formulation introduces the need for something like an entropy condition to achieve unique distribution solutions.  For example, in the case of a negative delta-ring measure, we have a decreasing jump in $m$ which introduces the possibility of a  rarefaction solution going forward in time.  In the classical weak solution formulation of Burgers equation, the entropy condition would select the rarefaction as the unique forward-time solution.  Nevertheless there exist other solutions, such as the outward-moving shock, that are bonafide distribution solutions, albeit ones that violate Lax's entropy condition whereby the speed of the shock should be faster than the characteristic speed ahead of it, and slower than the characteristic speed behind it.  

\section{Conclusions}

\

\label{conclusions}
We have considered existence of radially symmetric, monotone decreasing solutions to the aggregation equation in the case of more singular potentials $|x|^\alpha/\alpha$ for $\alpha$ in the range $2-d\leq \alpha <1$.
We remind the reader that the problem with $\alpha\geq1$ is known to be globally well-posed for measure data including the case without radial symmetry and monotonicity \cite{CDFLS}. 
For $2-d\leq \alpha <2$ we find that monotonicity is preserved, a feature that is not true for $\alpha > 2$.  Our results provide a rigorous framework for monotonicity behavior observed
in numerical simulations of finite time blowup \cite{HB, HB1} for radially symmetric 
data.  The results also provide
an understanding of the continuation of the solution after blowup.  That understanding
includes the result that one obtains instantaneous mass concentration for certain classes
of $L^1$ initial data including those observed as the asymptotic form of the blowup
profile in numerical simulations \cite{HB, HB1}.
 The special case of the Newtonian potential results in a localization of the problem, reducing to a form of the inviscid Burgers equation on the half line.  
In particular for radially symmetric decreasing data, there is a unique classical solution of the Burgers problem for all time, resulting in a unique solution of the original density problem.
This solution also retains its monotonicity.

In contrast to the Newtonian potential, for the case $1>\alpha>2-d$ the ensuing velocity field is at best H\"older continuous in time and our results are less precise. 
For example, uniqueness of solutions is still an open problem in this range, as is existence in the case of non-monotone, radially symmetric data.  The existence problem is complicated by the fact that the velocity field is at best H\"older continuous which makes it difficult to get convergence estimates for the flow map - something we use to prove existence of solutions in the case of 
monotone data.  It is somewhat ironic that the more singular case of the Newtonian potential
can be more easily solved - the velocity field is more singular, with a jump discontinuity.  However the localization of the dynamics results in a better understanding of the problem.  
For the general nonlocal problem in the range $2-d< \alpha <2$, the monotonicity assumption allows for smoother estimates on the velocity field, namely Lipschitz estimates, which allow
us to prove convergence of approximations and hence existence of a Lagrangian solution.
In addition to the above problems for data with symmetry, the general problem of measure solutions with non-radially symmetric data is wide open. 
Some insights can be gained from recent work on special families of weak solutions.
In the case of the Newtonian potential there exists a class of `patch solutions' that are the time-dependent characteristic functions of of a domain in $\real^d$.
These solutions have recently been observed \cite{patch} to converge in finite time to a measure supported on a set of codimension one.
Other works considers the analogue of vortex sheets for general aggregation equations with potentials that include both attraction and repulsion \cite{James,Hui,PRE}.


 \section*{Acknowledgments}  We thank the referees for many helpful comments.

\section{Appendix}
\subsection{Some general ODE results} \label{ode-appendix}

\

In standard ODE textbooks such as \cite{chicone}, it is proven that the flow map associated to a velocity field which is continuously differentiable in both space and time is itself differentiable. In our case of interest the velocity field is continuously differentiable in space but only continuous in time. We show here that the hypothesis of continuous differentiability in time can be replaced by
a weaker assumption that holds true in our case. Only very minor modifications are needed compared to standard proofs found in ODE textbooks such as \cite{chicone}. We will refer to \cite{chicone} and we will indicate the necessary modifications to be made in the proof there.

Let $\Omega\subset \real^d$ and $J \subset (-\infty,+\infty) $ be two open sets.  Suppose the function $v: \Omega \times J \to \real^d$ satisfies the following:
\begin{enumerate}
\item[(H0)] $v$ is continuous on $\Omega \times J$.
\item[(H1)] For every $t \in J$, the function $x\mapsto  v(x,t)$ is continuously differentiable on $\Omega$. 
\item[(H2)]  Given compact sets  $\bar{\Omega}_1 \subset \subset \Omega$ and $\bar{J}_1 \subset \subset J$, there exists $C>0$ such that
$$
\abs{\nabla v(x,t)} \le C
$$
for all $(x,t) \in \bar{\Omega}_1 \times \bar{J}_1$.
\item[(H3)] Given compact sets  $\bar{\Omega}_1 \subset \subset \Omega$ and $\bar{J}_1 \subset \subset J$, and given  $\epsilon>0$,  there exists $\delta>0$ such that
$$   \abs{\nabla v(y,t)-\nabla v(x,t)} \le \epsilon  $$
for all $x,y \in \bar{\Omega}_1$ satifying $\abs{x-y}< \delta$ and for all $t \in \bar{J}_1$.
\end{enumerate}
 In all the above $\nabla v(x,t)$ always stands for the derivative of $v$ with respect to $x$.
\begin{theorem} \label{chicone}
Under the above hypothesis, given $(x_0,t_0)\in \Omega \times J$ there exists open sets $\Omega_0 \subset \Omega$ and $J_0 \subset J$ such that $(x_0,t_0) \in \Omega_0 \times J_0$ and a unique continuous function $\sigma: \Omega_0 \times J_0 \mapsto \real^d$ such that
\begin{equation}\label{ode}
 \sigma(x,t)=x + \int_{t_0}^t v(\sigma(x,s),s) ds.
\end{equation}
 Moreover, given $t \in J_0$, the mapping $x \mapsto \sigma(x,t)$ is continuously differentiable on $\Omega_0$ and we have
 \begin{equation}\label{ode2}
 \nabla\sigma(x,t) = Id+  \int_{t_0}^t \nabla v( \sigma(x,s),s) \nabla\sigma(x,s) ds.
\end{equation}
\end{theorem}
\begin{proof} The proof of Theorem 1.184, page 120 in \cite{chicone} (or Theorem 1.261, page 138 of the online
version of \cite{chicone}) can be carried out  with very minor modifications. Let us just mentioned  where and how hypothesis (H3) is needed. 
 In the proof of \cite{chicone} 
  the space $X$ and $Y$ are defined by $X= C(b(t_0, \delta) \times B(x_0,\nu/2), \bar{B}(x_0,\nu))$ and $Y= C_b(b(t_0, \delta) \times B(x_0,\nu/2), L(\real^d, \real^d))$, where  $b(t_0, \delta)$ and $B(x_0,\nu/2)$ denotes balls of radius $\delta$ and $\nu/2$ and $L(\real^d, \real^d)$ denotes the set of linear transformations on $\real^n$. Both $X$ and $Y$ are endowed with the sup norm. $C_b$ stands for continuous and bounded. The mapping $\Psi: X \times Y \to Y$ is defined by
$$
\Psi(\phi,\Phi)(x,t)= Id+  \int_{t_0}^t \nabla v( \phi(x,s),s) \Phi(x,s) ds.
$$
In order to use the fiber contraction principle from \cite{chicone}, we must verify that  $\Psi$ is continuous.
From (H2) we easily obtain $\norm{\Psi(\phi,\Phi_1)-\Psi(\phi,\Phi_2)} \le K \delta \norm{\Phi_1-\Phi_2}$ where $K=\sup_{ B(x_0,\nu/2) \times b(t_0, \delta)} \abs{\nabla v}$. Hypothesis (H3) is needed in order to obtain continuity of $\Psi$ with respect to its first argument. To see this write
\begin{align}
\Psi(\phi_1,\Phi)(x,t)-\Psi(\phi_2,\Phi)(x,t) =  \int_{t_0}^t \p{ \nabla v( \phi_1(x,s),s)-\nabla v_2( \phi(x,s),s) } \Phi(x,s) ds.
\end{align}
Using  (H3) we see that  $\norm{\Psi(\phi_1,\Phi)-\Psi(\phi_2,\Phi)}$ can be made as small as we want by choosing $\phi_1$ and $\phi_2$ close enough with respect to the sup-norm.

\end{proof}
\begin{remark}
Since $\nabla v$ is not  assumed to be continuous with respect to time, the function   $t \mapsto \nabla\sigma(x,t)$ is not necessarily continuously differentiable. However it is  absolutely continuous as can been seen from   \eqref{ode2}. Therefore, given $x \in \Omega_0$, the  function $Y(t)= \nabla\sigma(x,t)$ is differentiable for almost every $t \in J_0$ and the differential equation 
$$
Y'(t)= \nabla v( \sigma(x,t),t)\;  Y(t)
$$
holds almost everywhere in $J_0$.  Then we can use Liouville Theorem (which is stated below) to deduce that \begin{equation}\label{ode50}
 \frac{d}{dt}\text{det }\nabla\sigma(x,t) = (\text{div }v)( \sigma(x,t),t) \;  \text{det } \nabla\sigma(x,t) 
\end{equation}
also holds almost everywhere in $J_0$. This of course implies that 
\begin{equation}\label{ode55}
 \text{det }\nabla\sigma(x,t) =   \text{exp} \big( {\int_{t_0}^t (\text{div }v)( \sigma(x,s),s) ds} \big),
\end{equation}
which is the formula needed in our case (see \eqref{ode555}).
\end{remark}

\begin{theorem}[Liouville] Let $A$ be  $d \times d$ matrix and
let Y(t) be a $d \times d$ time dependent matrix which is differentiable at $t=t_0$ and satisfies
$
Y'(t_0)= A \;  Y(t_0)$.
Then the function $\Lambda(t) = \text{ det } Y(t) $ is differentiable at $t_0$ and satisfies
$
\Lambda'  (t_0) = (\text{Tr }  A) \;   \Lambda (t_0). 
$
\end{theorem}
\begin{proof}
See for example Hartman \cite{Hartman}.
\end{proof}


%

\subsection{Some general Lemmas}

\


A proof of the following Lemma can be found in \cite[Lemma 3.11]{Canizo}.
\begin{lemma}\label{canizo} Let $T,S: \real^d \to \real^d$ be  two Borel maps. Also take $\rho \in \prob_2(\real^d)$. Then
$$
W_2(S\# \rho, T\# \rho) \le \norm{S-T}_{L^\infty(\text{supp} \rho )}.  
$$  
\end{lemma}

\begin{lemma} \label{preserve} Suppose that $\rho\in\prob(\real^d)$ has  compact support and
suppose that  $\prob_{RD}(\real^d) \ni \rho_n$ converges narrowly to $\rho$. Then $\rho$ also belongs to $\prob_{RD}(\real^d)$.
\end{lemma}
\begin{proof}  Let $R$ be a rotation of $\real^d$ and let $f \in C(\real^d).$
Then $f \circ R \in C(\real^d)$ and $\int f \circ R d\rho_n = \int f d\rho_n$.
Taking limits we see that $\int f\circ R d\rho = \int f d\rho$ so that
$\rho \in  
\mes_R(\real^d).$  To prove $\rho$ is decreasing fix $0 < r_1 < r_2$ and
take disjoint small rings $A_j = \{r_j - \eta_j \leq |x| \leq r_j + \delta_j\}, j = 1,2$
having the same volume. We may assume $\rho(\partial A_j) =0$.  Then
there exist continuous functions $f_1 \geq \chi_{A_1}$ and $f_2 \leq
\chi_{A_2}$ with disjoint supports such that
 $|\rho(A_j) - \int f_j d\rho| < \epsilon.$
By hypothesis we have $\inf f_1 d\rho_n \geq \int f_2 d\rho_n$,
so that $\rho(A_2) + 2\epsilon \leq \rho(A_1).$  Then shrinking $\epsilon,$
$A_1$ and $A_2$ shows that $\rho \in \prob_{RD}(\real^d).$
\end{proof}

\

\begin{lemma} \label{holderlem}
Let $\alpha > 2-d$ and  suppose $y: [0,+\infty) \to [0,+\infty)$ is an absolutely continuous function satisfying
$- y(t)^{\alpha-1} \le y'(t) \le 0$ for almost every $t \in [0,+\infty)$ for which $y(t)>0$.
  If $\alpha \le 1$ then  $y(t)$ is H\"older continuous. To be more precise: 
  $$
- ( (2-\alpha)(t-s))^\frac{1}{2-\alpha} \le y(t)-y(s) \le 0
  $$ 
  for all $0\le s \le t$.  If $\alpha>1$ then $y(t)$ is Lipschitz continuous. To be more precise: 
  $$
  -y(0)^{\alpha-1} (t-s) \le y(t)-y(s) \le 0
  $$ 
  for all $0\le s \le t$.
\end{lemma}
\begin{proof} The case $\alpha>1$ is trivial since  the inequality $- y(t)^{\alpha-1} \le y'(t) \le 0$ together with the non-negativity of $y$ implies
  $- y(0)^{\alpha-1} \le y'(t) \le 0$. We now prove the Lemma for $2-d<\alpha \le 1$.
For almost every $t\ge 0$ for which $y(t)>0$ we have $- (2-\alpha) \le \frac{d}{dt}\p{ y(t)^{2-\alpha}} \le 0$. It is then clear that
 $ -(2-\alpha)(t-s) \le y(t)^{2-\alpha}-y(s)^{2-\alpha} \le 0$ for all $0\le s \le t$. But because of the convexity of the function $r \mapsto r^{2-\alpha}$ we have that $(y(s)-y(t))^{2-\alpha} \le y(s)^{2-\alpha}-y(t)^{2-\alpha}$, which gives the result.\end{proof}

\end{document}